\documentclass[11pt,twoside,reqno]{amsart}
\usepackage{amsthm, amsmath, amscd, amssymb,centernot}
\usepackage{enumitem}
\usepackage[T1]{fontenc}
\usepackage[left=2.5cm,top=2.5cm,bottom=3cm,right=2.5cm]{geometry}
\usepackage[pagebackref]{hyperref}
\usepackage[capitalise]{cleveref}
\usepackage{tikz-cd}

\usepackage{listings}
\usepackage{xcolor}

\theoremstyle{definition}
\newtheorem{Defn}{Definition}[section]

\newtheorem{thm}[Defn]{Theorem}
\newtheorem*{thm*}{Theorem}
\newtheorem{prop}[Defn]{Proposition}

\newtheorem{lemma}[Defn]{Lemma}
\newtheorem{remark}[Defn]{Remark}
\newtheorem{exm}[Defn]{Example}

\author[Shilpa Rani]{Shilpa Rani}
 \address{Department of Mathematics, SRM University AP, Amaravati 522 240, Andhra Pradesh,India}
 \email{ranishilpa7377@gmail.com}
 \title{Surjectivity of certain Word maps in $SU(2)$ and $SL(2,\mathbb{C})$}
 
\begin{document}

    	\begin{abstract}
   	    In this article, we show the surjectivity  of word maps $w:SU(2) \times SU(2) \to SU(2)$ induced by several families of words in the free group of rank 2. Also, we prove the surjectivity of certain word maps on $SL(2,\mathbb{C})$.
   	\end{abstract}
\maketitle
\section{Introduction}

  Let \( \mathbf{F}_2 \) denote the free group on two generators \( a \) and \( b \), and let \( G \) be an arbitrary group. For any word \( w(a,b) \in \mathbf{F}_2 \), one defines a map $w : G \times G \to G, $ $ (x,y) \mapsto w(x,y)$,
called the \emph{word map} associated with the word \( w \).
	The study of word maps was initially motivated by the commutator maps and power maps which are some special cases of word maps that have a larger impact on the study of groups. The field has expanded to include a wide range of problems linking group theory, geometry, and number theory. Word maps reveal structural properties of groups and have led to several deep theorems, while many important questions and conjectures remain open, among which the surjectivity of word maps is a central problem. 
    
 	It can be easily seen that the image of a word map is invariant under conjugation and it always contains the identity element. In recent years, considerable progress has been made in the study of word maps, particularly, in the case of finite groups (see \cite{ML1,ML2,ML}). One of the central conjectures in this area is the \textit{Ore conjecture} \cite{OO}, which asserts that the word map corresponding to the commutator word is surjective for every non-abelian finite simple group. Significant contributions were made toward this conjecture by Thompson who proved it for $\text{PSL}(n,\mathbb{F}_{p})$ (see \cite{RCT}). The conjecture was eventually settled  in 2010 by M. W. Liebeck, E. A. O'Brien, A. Shalev, and P. H. Tiep \cite{MWLEAO}.  

    Beyond finite groups, related questions have also been studied in the setting of algebraic groups. In 2002-03, P. Chatterjee \cite{PC1, PC2} established that for a connected semisimple algebraic group, the power map is mostly surjective. Independently, Steinberg \cite{RS} obtained a similar result using a different method.   A notable result supporting the surjectivity of word maps is that the word map
$w : PSL(2,\mathbb{C}) \times PSL(2,\mathbb{C}) \to PSL(2,\mathbb{C})$
is surjective for all \( w \in \mathbf{F}_2^{(1)} \setminus \mathbf{F}_2^{(2)} \), where \( \mathbf{F}_2^{(n)} \) denotes the \(n\)th term of the derived series of \( \mathbf{F}_2 \) (see \cite{TBYGZ}).
Furthermore, the word map corresponding to any nontrivial double commutator word is surjective on \( PSL(2,K) \), where \( K \) is an algebraically closed field of characteristic \(0\) (see \cite{UTJS}).

	Another important result about the word maps is that the image of every word map is dense for a connected semisimple algebraic group over an algebraically closed field \cite{AB}. It can be easily seen that there are words for which the corresponding word maps are not surjective. 
	\begin{exm}
		Take $G=SL(2,\mathbb{C})$ and $w(a,b)=a^{2}$. The element $x=\begin{bmatrix} -1 & 1 \\ 0 & -1 \end{bmatrix}$ has no preimage in $SL(2,\mathbb{C})$.
	\end{exm}

    At the 2008 Spring Central Section Meeting of the AMS in Bloomington, 
    Michael Larsen posed the question that the word map for all $w\in \textbf{F}^{(n)}$ for sufficiently large	$n\in \mathbb{N}$ is surjective.  
	It is very easy to show that $w\notin \textbf{F}_{2}^{(1)}$, $w:SU(n)\times SU(n)\rightarrow SU(n)$ is surjective \cite[Lemma 2.1]{AEAT}. Goto \cite{MGo} and Toyama \cite{HTo} proved that the commutator word map on $SU(n)$ is surjective for all $n\in \mathbb{N}$.  

    Drawing motivation from the recent work of Vu The Khoi and Ho Minh Toan~\cite{VKHMT}, in which they establish the surjectivity of certain families of word maps on \( $SU(2)$ \), we investigate the surjectivity of related classes of word maps. The main results of this paper are the following.

   
\begin{thm}\label{A}
    
 The word map on \(SU(2)\) induced by the word	$w_n \in \mathbf{F}_{2}$	is surjective  whenever \(w_n\) takes one of the following forms:
		\begin{itemize}
			\item [(a)]	$	[[a,b],a^{n}[a,b]a^{-n}] $, $n\geq 1$
			\item [(b)]  $[[a,b],b^{n}[a,b]b^{-n}]$, $n\geq 1$
			\item [(c)]	$[[a,b],(ab)^{n}[a,b](ab)^{-n}]$, $n\geq 1$
			\item [(d)]	$[[a,b],a^{n}b[a,b]b^{-1}a^{-n}]$, $n\geq 2$
			\item[(e)] $[[a,b],a^{n}b^{2}[a,b]b^{-2}a^{-n}]$, $n\geq 1$
			\item [(f)] $[[a,b],a^{n}b^{m}[a,b]b^{-m}a^{-n}]$, $n\geq 2, m\geq 2 $. When $n\equiv 1,2,4,5 $ (mod $6$) and $m\equiv 0,2$ (mod $4$). 
            \item[(g)] $[[a,b],[a,(ab)^{n}]]$, $n\geq 2$.

		\end{itemize}
		
	\end{thm}
    
		
	\begin{thm}\label{B}
	The word map $SL(2,\mathbb{C})\times SL(2,\mathbb{C})\rightarrow SL(2,\mathbb{C}) $ induced by any of the following words are surjective,	
		\begin{center}
				\begin{itemize}
				\item $w_{1}(a,b)=[[a,b],a[a,b]a^{-1}]$
				\item $w_{2}(a,b)=[[a,b],b[a,b]b^{-1}]$
				\item $w_{3}(a,b)=[[a,b],(ab)[a,b](ab)^{-1}]$
				\item $w_{4}(a,b)=[[a,b],ab^{2}[a,b]b^{-2}a^{-1}]$
				\item $w_{5}(a,b)=[[a,b],a].$
				
			\end{itemize}
		\end{center} 
	\end{thm}

	\section{Preliminaries}

    In this section, we recall some results that are needed in the  sequel.
    
	  The following facts about $SU(2)$ and the trace of a matrix in $SU(2)$ are immediate.
          \begin{enumerate}
	  	\item 	A matrix in $SU(2)$ is in the form,
	  $\begin{bmatrix}
	  		a & b \\
	  		-\overline{b} & \overline{a}
	 	\end{bmatrix}  $ where $|a|^{2}+|b|^{2}=1$.
	 	\item  The trace of $A$, $T(A)\in [-2,2]$.
	  	\item $A$ satisfies $A^{2}-T(A) A+I=0$. 
		
	  	\item $T(A)=T(A^{-1})$
	  	\item $	T(AB)+T(AB^{-1})=T(A)T(B)$
	  	\item $	T(ABA^{-1}B^{-1})=T(A)^{2}+T(B)^{2}+T(AB)^2-T(A)T(B)T(AB)-2.$
	 \end{enumerate}

    The Chebyshev polynomials are  important ingredients that help in our study of word maps, in particular the second Chebyshev polynomials. We define the second Chebyshev polynomials $U_{n}$  by iteration:

    \begin{Defn}

		Set $U_{0}(x)=1, U_{1}(x)=2x$.  Then the Second Chebyshev polynomial $U_{n}$ is defined as $U_{n+1}(x)=2xU_{n}(x)-U_{n-1}(x)$.
        
    \end{Defn}
    
    \begin{remark}

    When $x=\mathrm{\mathrm{cos}} \theta$ the second Chebyshev polynomials have the trigonometric form 
    \begin{center}
         $U_n(x)=\frac{\mathrm{sin}(n+1) \theta}{\mathrm{sin} \theta}$.
    \end{center}
         So we have,\[ 
	U_{n}(0) = \left\{
	\begin{array}{ll}
		0 & \text{if } n\equiv 1,3 \text{(mod4)} \\
		1 & \text{if } n\equiv 0 \text{(mod4)} \\
		-1 & \text{if }  n\equiv 2 \text{(mod4)}.
	\end{array}
	\right.
	\] 
        \end{remark}
	Now we list some lemmas that will be useful in later sections.
	\begin{lemma}\label{L1}\cite[Lemma 2.1, Lemma 2.3]{VKHMT} We have following relations:
    \begin{enumerate}
        \item Let $A\in SU(2)$. If $x=  T(A)$, then $A^{n}=U_{n-1}(\frac{x}{2})A-U_{n-2}(\frac{x}{2})I$, for all integers $n\geq 2$.
        \item $U_{n}(x)^{2}+U_{n-1}(x)^{2}-2xU_{n}(x)U_{n-1}(x)=1$ for all $n \geq 1.$
    \end{enumerate}
	\end{lemma}
Consider the action of $SU(2)$ on $(SU(2)\times SU(2))$ by $X.(A,B)=(XAX^{-1},XBX^{-1})$. The orbit space $(SU(2)\times SU(2))/SU(2)$ can be identified 
     $$\tau = \{ (x,y,z)\in [-2,2]^{3}\mid | k(x,y,z)| \leq2\},\text{where}       ~ k(x,y,z)=x^{2}+y^{2}+z^{2}-xyz-2,$$ 
     by the map,
	$$(A,B)\longmapsto (T(A),T(B),T(AB)).$$ 
	
	For any word $w$, the trace of $w(A,B)$ can be expressed as a polynomial $p_{w}$  in the variables $x=T(A), y=T(B), z=T(AB)$ with integer coefficients \cite{VH},\cite{FRKF}.  Since $T(ABA^{-1}B^{-1})=T(A)^{2}+T(B)^{2}+T(AB)^2-T(A)T(B)T(AB)-2$, for the word $w(A,B)=ABA^{-1}B^{-1}$ we can say $p_{w}$ as the polynomial $k(x,y,z)=x^{2}+y^{2}+z^{2}-xyz-2$.
	\begin{lemma} \cite[Proposition 2.1]{VKHMT} \label{L2}
		Let $w\in \textbf{F}_{2}^{(1)}$ be a word of the form $w(a,b)=[u(a,b),v(a,b)]$ where $u,v\in \textbf{F}_{2}$. Then the word map $w$ is surjective if and only if the system of equations $\{ p_{u}=0,  p_{v}=0,  p_{uv}=0 \}$ has a solution in $\tau $.
	\end{lemma}
	
	\begin{remark}
		As \( SO(2, \mathbb{C}) \) is abelian, it is clear that no word \( w \in \mathbf{F}_{2}^{(1)} \) is surjective.
		
		Its maximal torus is the group itself, and each element in $SO(2,\mathbb{C})$ is conjugate to a matrix of the form:
		\[
		\begin{bmatrix} \alpha & 0 \\ 0 & \alpha^{-1} \end{bmatrix}, \quad \text{where } \alpha \in \mathbb{C}^{\ast}.
		\]
		
		Let \( w \in \mathbf{F}_{2} \setminus \mathbf{F}_{2}^{(1)} \) be a word of the form:
		\[
		w = a^{n_1} b^{m_1} a^{n_2} b^{m_2} \cdots a^{n_t} b^{m_t}
		\]
		with either \( \sum_{i=1}^{t} n_i \ne 0 \) or \( \sum_{i=1}^{t} m_i \ne 0 \). Without loss of generality, assume \( N := \sum_{i=1}^{t} n_i \ne 0 \).
		
		Let
		\[
		a = \begin{bmatrix} x & 0 \\ 0 & x^{-1} \end{bmatrix}, \quad b = I = \begin{bmatrix} 1 & 0 \\ 0 & 1 \end{bmatrix}.
		\]
		Then the word map evaluates to:
		\[
		w(a, b) = \begin{bmatrix} x^{N} & 0 \\ 0 & x^{-N} \end{bmatrix}.
		\]
		
		Now, if we choose \( x \) to be an \( N \)th root of any \( \alpha \in \mathbb{C}^* \), then:
		\[
		w(a, b) = \begin{bmatrix} \alpha & 0 \\ 0 & \alpha^{-1} \end{bmatrix}.
		\]
		
		Hence, every matrix of this form lies in the image of the word map.  The word is surjective, since $w(PAP^{-1},PBP^{-1})=Pw(A,B)P^{-1}$. So, every word in \( \mathbf{F}_{2} \setminus \mathbf{F}_{2}^{(1)} \) is surjective on \( SO(2, \mathbb{C}) \).
	\end{remark}
	
	\section{Surjectivity of word maps in $SU(2)$}
	In this section, we will show the surjectivity of some family of words on $SU(2)$.
	\begin{prop}\label{a1}
		The word map induced by the word $w(a,b)=[[a,b],a[a,b]a^{-1}]$ is surjective on $SU(2)$.
	\end{prop}
	\begin{proof}
		We prove this using the Lemma \ref{L2}. We know that polynomial corresponding to the words $[a,b]$ and $ a[a,b]a^{-1}$ is $k(x,y,z)$. That is,
		\begin{center}
			$p_{[a,b]}=k(x,y,z)$, \\
			$p_{a[a,b]a^{-1}}=k(x,y,z).$ 
		\end{center}
Let  $ A,B \in SU(2)$. Consider,\\
		$T([A,B]A[A,B]A^{-1})= T(ABA^{-1}B^{-1}AABA^{-1}B^{-1}A^{-1})
		= T(A^{-1}B^{-1}A^{2}BA^{-1}).$
\begin{equation*}
			\begin{aligned}
				A^{-1}B^{-1}A^{2}BA^{-1}&=A^{-1}B^{-1}(T(A)A-I)BA^{-1}\\
				&= T(A)^{2}A^{-1}B^{-1}AB-T(A)A^{-1}B^{-1}ABA-T(A)A^{-1}+I.
			\end{aligned}
		\end{equation*}
		Thus, 	
		\begin{equation*}
			\begin{aligned}
				T([A,B]A[A,B]A^{-1})&= T(T(A)^{2}A^{-1}B^{-1}AB-T(A)A^{-1}B^{-1}ABA-T(A)A^{-1}+I)\\&= T(A)^{2}T(A^{-1}B^{-1}AB)-2T(A)T(A)+2.
			\end{aligned}
		\end{equation*}
		Then the polynomial corresponding to the word $[a,b]a[a,b]a^{-1}$ is,
		\begin{center}
			$p_{[a,b]a[a,b]a^{-1}}= x^{2}k(x,y,z)-2x^{2}+2.$
		\end{center}
        Now consider the system of equations,
		\begin{equation}\label{a.1}
			\begin{aligned}
				&p_{[a,b]}=0,\\ &p_{a[a,b]a^{-1}}=0,\\ &p_{[a,b]a[a,b]a^{-1}}=0.
			\end{aligned}
		\end{equation}
		Now, using	$p_{[a,b]}=k(x,y,z)=0$ the third equation in the system (\ref{a.1}) becomes $-2x^{2}+2=0$.
		That is, $x= 1$ or $-1$. This shows that $(1,1,0)\in \tau $ is a solution to the system (\ref{a.1}). Therefore, the word map corresponding to the given word $w $ is
surjective.
	\end{proof}
	
	\newpage
	\begin{prop}\label{a2}
		The word map induced by the word $w(a,b)=[[a,b],a^{n}[a,b]a^{-n}]$, $n\geq 2$ is surjective on $SU(2)$.
	\end{prop}
	\begin{proof}
		Let  $ A,B \in SU(2)$. Using (1) of Lemma \ref{L1},
		\begin{equation*}
			\begin{aligned}
				  [A,B]A^{n}[A,B]A^{-n}&= [A,B](U_{n-1}(x/2)A-U_{n-2}(x/2)I)[A,B](U_{n-1}(x/2)A^{-1}-U_{n-2}(x/2)I)\\ &= U_{n-1}(x/2)^{2}[A,B]A[A,B]A^{-1}-U_{n-1}(x/2)U_{n-2}(x/2)[A,B]A[A,B]\\&\hspace*{0.5cm}-U_{n-1}(x/2)U_{n-2}(x/2)[A,B][A,B]A^{-1}+ U_{n-2}(x/2)^{2}[A,B]^{2}.
			\end{aligned}
		\end{equation*}
		Hence we get,
		\begin{equation}\label{1.1}
			\begin{aligned}
				T([A,B]A^{n}[A,B]A^{-n})&=U_{n-1}(x/2)^{2}T([A,B]A[A,B]A^{-1})-U_{n-1}(x/2)U_{n-2}(x/2)T([A,B]A[A,B])\\& \hspace*{0.5cm} -U_{n-1}(x/2)U_{n-2}(x/2)T([A,B][A,B]A^{-1})+ U_{n-2}(x/2)^{2}T([A,B]^{2}).
			\end{aligned}
		\end{equation}
		And
		\begin{equation}\label{1.2}
			\begin{aligned}
				T([A,B]A[A,B]A^{-1})&= 
                T(A)^{2}T([A,B])-2T(A)^{2}+2.
			\end{aligned}
		\end{equation}
		Also,
		\begin{equation}\label{1.3}
			\begin{aligned}
				T([A,B]A[A,B])&= T((T([A,B])[A,B]-I)A)\\&=T([A,B])T([A,B]A)-T(A),
			\end{aligned}
		\end{equation}
		\begin{equation}\label{1.4}
			\begin{aligned}
				T([A,B][A,B]A^{-1})&= T([A,B])T(A)-T(A).
			\end{aligned}
		\end{equation}
		Substituting (\ref{1.2}),(\ref{1.3}) and (\ref{1.4}) in (\ref{1.1}), we get;
		\begin{equation}
			\begin{aligned}
				T([A,B]A^{n}[A,B]A^{-n})&=U_{n-1}(x/2)^{2}(T(A)^{2}T([A,B])-2T(A)^{2}+2)\\&\hspace*{0.5cm}-U_{n-1}(x/2)U_{n-2}(x/2)(T([A,B])T([A,B]A)-T(A))\\& \hspace*{0.5cm} -U_{n-1}(x/2)U_{n-2}(x/2)(T([A,B])T(A)-T(A))+ U_{n-2}(x/2)^{2}T([A,B]^{2}).
			\end{aligned}
		\end{equation}
		Now,
		\begin{equation*}
			\begin{aligned}
				p_{[a,b]a^{n}[a,b]a^{-n}}&=U_{n-1}(x/2)^{2}(x^{2}k(x,y,z)-2x^{2}+2)\\&\hspace*{0.5cm}-U_{n-1}(x/2)U_{n-2}(x/2)(xk(x,y,z)^{2}-xk(x,y,z)-x)\\& \hspace*{0.5cm} -U_{n-1}(x/2)U_{n-2}(x/2)(xk(x,y,z)-x)+ U_{n-2}(x/2)^{2}(k(x,y,z)^{2}-2),\\
				p_{[a,b]}&=k(x,y,z),\\
				p_{a^{n}[a,b]a^{-n}}&= k(x,y,z).
			\end{aligned}
		\end{equation*}
        Now consider the system,
		\begin{equation}\label{2}
			\begin{aligned}
				p_{[a,b]}&=0,\\ p_{a^{n}[a,b]a^{-n}}&=0,\\ p_{[a,b]a^{n}[a,b]a^{-n}}&=0.
			\end{aligned}
		\end{equation}
		That is,
		\begin{center}
			$k(x,y,z)=0,$
		\end{center}
		
			$x^{2}U_{n-1}(x/2)^{2}-2U_{n-1}(x/2)^{2}+U_{n-1}(x/2)^{2}+U_{n-2}(x/2)^{2}-xU_{n-1}(x/2)U_{n-2}(x/2)=0. $\\
		\\
		Using the relation $U_{n}(x)^{2}+U_{n-1}(x)^{2}-2xU_{n}(x)U_{n-1}(x)=1$, the system becomes;
		\begin{center}
        $k(x,y,z)=0,$\\
			$x^{2}U_{n-1}(x/2)^{2}-2U_{n-1}(x/2)^{2}+1=0 $
		\end{center}
		So system (\ref{2}) is equivalent to the following system,
		\begin{equation}\label{3.a}
			\begin{aligned}
				k(x,y,z)&=0,
				\\ 1+(x^{2}-2)U_{n-1}(x/2)^{2}&=0.
			\end{aligned}
		\end{equation}
		When $n=2$, $1+(x^{2}-2)x^{2}=0 \Rightarrow x^{4}-2x^{2}+1=0 \Rightarrow x^{2}=1$. Therefore, $(1,1,0)\in \tau $ is a solution to the system. 
        When $n=4$, $x=-1$ is a root of $1+(x^{3}-2x)(x^{2}-2)=0$. So $(-1,1,0)$ is a solution to the system in $\tau$.
        
		In general, when $n$ is odd, $(n-1)$ is even and hence $U_{n-1}(0)^{2}=1$. Considering the second equation of the system (\ref{3.a}) at $0$, it will give $-1 < 0$. Also, when evaluating the same equation at $\sqrt{2}$ we get $1 > 0$. Hence, we conclude that the second equation has a root between $\sqrt{2}$ and $0$. Say $\alpha$. We can find a $\beta$ such that $\alpha^{2}+\beta^{2}=2$ in $[-2,2]$. Thus, $(\alpha ,\beta ,0)\in \tau$ is a solution to system (\ref{3.a}). 

		When  n is even, take $x=2\mathrm{cos}(\frac{(n-2)\pi }{2(n-1)})$.\\
		Then, \begin{equation*}
			\begin{aligned}
				U_{n-1}\left(\frac{x}{2}\right)=	U_{n-1}\left(\mathrm{cos}\left(\frac{(n-2)\pi }{2(n-1)}\right) \right)&=\frac{\mathrm{sin}(\frac{n(n-2)\pi }{2(n-1)})}{\mathrm{sin}(\frac{(n-2)\pi }{2(n-1)})}=\frac{\mathrm{sin} (\frac{(n-2)\pi }{2}+\frac{(n-2)\pi }{2(n-1)})}{\mathrm{sin}(\frac{(n-2)\pi }{2(n-1)})}\\ &=\frac{\mathrm{sin}(\frac{(n-2)\pi }{2}) \mathrm{cos}(\frac{(n-2)\pi }{2(n-1)})-\mathrm{cos} (\frac{(n-2)\pi }{2})\mathrm{sin}(\frac{(n-2)\pi }{2(n-1)})}{\mathrm{sin} (\frac{(n-2)\pi }{2(n-1)}).}
			\end{aligned}
		\end{equation*}
		Thus when $n$ is even, we get $\mathrm{sin}(\frac{(n-2)\pi }{2})=0 $ and $\mathrm{cos}(\frac{(n-2)\pi }{2})=(-1)^{\frac{n-2}{2}}$.\\
		Hence, $U_{n-1}(\frac{x}{2})^{2}=1$.\\
		Thus the second equation of the system (\ref{3.a}) becomes, \begin{center}
     $1+(x^{2}-2)U_{n-1}(x/2)^{2}=			1+4\mathrm{cos}^{2}\left(\frac{(n-2)\pi}{2(n-1)} \right)-2$\\ $= 4\mathrm{cos}^{2}\left(\frac{(n-2)\pi}{2(n-1)} \right)-1$.\\ $4\mathrm{cos}^{2}\left(\frac{(n-2)\pi}{2(n-1)} \right)-1<0 \Leftrightarrow \mathrm{cos}^{2}\left(\frac{(n-2)\pi}{2(n-1)} \right)< \frac{1}{4}$.
		\end{center}
		Since $\mathrm{cos}(x)$ is decreasing in $[0,\pi]$,
		\begin{center}
			$\mathrm{cos}\left(\frac{(n-2)\pi}{2(n-1)} \right)<\frac{1}{2}=\mathrm{cos}\left(\frac{\pi}{3}\right)\Leftrightarrow  \frac{(n-2)\pi}{2(n-1)} > \frac{\pi}{3} \Leftrightarrow 3n-6 > 2n-2 \Leftrightarrow n>4.$
		\end{center}
        
		Hence we get that, when $n>4$ the second equation in the system (\ref{3.a}) is negative for $x=2\mathrm{cos}(\frac{(n-2)\pi }{2(n-1)})$. Hence there is a root $x$ such that $x<\sqrt{2}$, which gives the existence of  $y,z\in [-2,2]$ such that $k(x,y,z)=0$. 
        
		So, system (\ref{2}) has solution in $\tau$. Therefore, the word map corresponding to the given word $w$ is surjective on $SU(2)$.
		\end{proof}

\newpage
	\begin{prop}\label{a3}
		The word map induced by the word $w(a,b)=[[a,b],b[a,b]b^{-1}]$ is surjective on $SU(2)$.
	\end{prop}
	\begin{proof}
	For any $A,B \in SU(2)$ we have,
		\begin{equation*}\label{4.1a}
			\begin{aligned}
				[A,B]B[A,B]B^{-1}&=ABA^{-1}B^{-1}BABA^{-1}B^{-1}B^{-1}\\ &=AB^{2}A^{-1}B^{-2}\\ &= T(B)^{2}ABA^{-1}B^{-1}-T(B)ABA^{-1}-T(B)B+I.	
			\end{aligned}
		\end{equation*}
        Therefore, trace is
        \begin{equation}
            T([A,B]B[A,B]B^{-1})=T(B)^{2}T(ABA^{-1}B^{-1})-2T(B)^{2}+2.
        \end{equation}
        Consider the system of equations,
		\begin{equation}\label{4}
			\begin{aligned}
				p_{[a,b]}&=0,\\ p_{b[a,b]b^{-1}}&=0,\\ p_{[a,b]a[a,b]a^{-1}}&=0.
			\end{aligned}
		\end{equation}
		From (\ref{4.1a}), the system (\ref{4}) becomes
		\begin{equation*}
			\begin{aligned}
				k(x,y,z)&=0,\\
				y^{2}k(x,y,z)-2y^{2}+2&=0.
			\end{aligned}
		\end{equation*}
        When $ -2y^{2}+2=0
				\Rightarrow y=1 \text{or} -1.$		
		So, $(1,1,0)\in \tau$ is solution to system (\ref{4}). Hence, the given word map is surjective.
	\end{proof}

	\begin{prop}\label{a4}
		The word map induced by the word $w(a,b)=[[a,b],b^{n}[a,b]b^{-n}]$,$n\geq 2$ is surjective on $SU(2)$.
	\end{prop}
	\begin{proof}
	Let $A,B \in SU(2)$. By virtue of (1) of Lemma \ref{L1}, 
		\begin{equation*}
			\begin{aligned}
				[A,B]B^{n}[A,B]B^{-n}&= [A,B](U_{n-1}(y/2)B-U_{n-2}(y/2)I)[A,B](U_{n-1}(y/2)B^{-1}-U_{n-2}(y/2)I)\\ &=(U_{n-1}(y/2)[A,B]B-U_{n-2}(y/2)[A,B])(U_{n-1}(y/2)[A,B]B^{-1}-U_{n-2}(y/2)[A,B])\\&= U_{n-1}(y/2)^{2}[A,B]B[A,B]B^{-1}-U_{n-1}(y/2)U_{n-2}(y/2)[A,B]B[A,B]\\&\hspace*{0.5cm}-U_{n-1}(x/2)U_{n-2}(y/2)[A,B][A,B]B^{-1}+ U_{n-2}(y/2)^{2}[A,B]^{2}.
			\end{aligned}
		\end{equation*}
		Hence we get,
		\begin{equation}\label{5.1}
			\begin{aligned}
				T([A,B]B^{n}[A,B]B^{-n})&=U_{n-1}(y/2)^{2}T([A,B]B[A,B]B^{-1})-U_{n-1}(y/2)U_{n-2}(y/2)T([A,B]B[A,B])\\& \hspace*{0.5cm} -U_{n-1}(y/2)U_{n-2}(y/2)T([A,B][A,B]B^{-1})+ U_{n-2}(y/2)^{2}T([A,B]^{2}).
			\end{aligned}
		\end{equation}
		And 
		\begin{equation}\label{5.2}
			\begin{aligned}
				T([A,B]B[A,B]B^{-1})&= T(B)^{2}T([A,B])-2T(B)^{2}+2.
			\end{aligned}
		\end{equation}
		Also,\begin{equation}\label{5.3}
			\begin{aligned}
				T([A,B]B[A,B])&= T((T([A,B])[A,B]-I)B)\\&=T([A,B])T([A,B]B)-T(B)\\&= T([A,B])T(B)-T(B),
			\end{aligned}
		\end{equation}
		
		\begin{equation}\label{5.4}
			\begin{aligned}
				T([A,B][A,B]B^{-1})&= T((T([A,B])[A,B]-I)B^{-1})\\&=  T([A,B])T(ABA^{-1}(T(B)B-I))-T(B) \\&= T(B)T([A,B])^{2}-T([A,B])T(B)-T(B).
			\end{aligned}
		\end{equation}
		Substituting (\ref{5.2}),(\ref{5.3}) and (\ref{5.4}) in (\ref{5.1}), we get;
		\begin{equation}
			\begin{aligned}
				T([A,B]B^{n}[A,B]B^{-n})&=U_{n-1}(y/2)^{2}(T(B)^{2}T([A,B])-2T(B)^{2}+2)\\&\hspace*{0.5cm}-U_{n-1}(y/2)U_{n-2}(y/2)(T([A,B])T(B)-T(B))\\& \hspace*{0.5cm} -U_{n-1}(y/2)U_{n-2}(y/2)(T(B)T([A,B])^{2}-T([A,B])T(B)-T(B))\\& \hspace*{0.7cm} +U_{n-2}(y/2)^{2}T([A,B]^{2}).
			\end{aligned}
		\end{equation}
		Now,\begin{equation*}
			\begin{aligned}
				p_{[a,b]b^{n}[a,b]b^{-n}}&=U_{n-1}(y/2)^{2}(y^{2}k(x,y,z)-2y^{2}+2)\\&\hspace*{0.5cm}-U_{n-1}(y/2)U_{n-2}(y/2)(yk(x,y,z)^{2}-yk(x,y,z)-y)\\& \hspace*{0.5cm} -U_{n-1}(y/2)U_{n-2}(y/2)(yk(x,y,z)-y)+ U_{n-2}(y/2)^{2}(k(x,y,z)^{2}-2),\\
				p_{[a,b]}&=k(x,y,z),\\
				p_{b^{n}[a,b]b^{-n}}&= k(x,y,z).
			\end{aligned}
		\end{equation*}
        Consider the system of equations,
		\begin{equation}\label{16}
			\begin{aligned}
				p_{[a,b]}&=0,\\ p_{b^{n}[a,b]b^{-n}}&=0,\\ p_{[a,b]b^{n}[a,b]b^{-n}}&=0.
			\end{aligned}
		\end{equation}
		The system (\ref{16}) equivalent to: \begin{center}
			$k(x,y,z)=0,$\\
            $U_{n-1}(y/2)^{2}(-2y^{2}+2)+2yU_{n-1}(y/2)U_{n-2}(y/2)-2 U_{n-2}(y/2)^{2}=0. $
		\end{center}
        By using the relation we have, $U_{n}(y)^{2}+U_{n-1}(y)^{2}-2yU_{n}(y)U_{n-1}(y)=1$.
		\begin{center}
			$y^{2}U_{n-1}(y/2)^{2}-2U_{n-1}(y/2)^{2}+1=0 $
			$\Rightarrow$ $1+(y^{2}-2)U_{n-1}(y/2)^{2}=0.$
		\end{center}
		So system (\ref{16}) is equivalent to the following system,
		\begin{equation*}\label{17}
			\begin{aligned}
				k(x,y,z)&=0,
				\\ 1+(y^{2}-2)U_{n-1}(y/2)^{2}&=0.
			\end{aligned}
		\end{equation*}
		The given system of equations corresponds to that in Proposition \ref{a2}, with the variable $x$ replaced by $y$. Hence, we can find the solution in similar way as shown in that proposition by substituting $x$ with $y$. Therefore, the given word map is surjective on $SU(2)$.
		\end{proof}

    \begin{prop}\label{a5}
		The word map induced by the word $w(a,b)=[[a,b],(ab)[a,b](ab)^{-1}]$ is surjective on $SU(2)$.
	\end{prop}
	\begin{proof}
    Let $A,B$ be matrices in $SU(2)$. 
		Consider,\begin{equation}\label{6.1}
			\begin{aligned}
				[A,B]AB[A,B](AB)^{-1}&=ABA^{-1}B^{-1}ABABA^{-1}B^{-1}B^{-1}A^{-1}.\\
			 \text{Therefore,}\hspace{.5cm} T([A,B]AB[A,B](AB)^{-1})&= T(A^{-1}B^{-1}ABABA^{-1}B^{-1})\\&= T((A^{-1}B^{-1})^{2}(AB)^{2}). 		\end{aligned}
		\end{equation}
        We have,  
		$(AB)^{2}=T(AB)AB-I$ and $(A^{-1}B^{-1})^{2}=T(AB)A^{-1}B^{-1}-I$.
        Thus,
		\begin{equation}\label{6.2}
			\begin{aligned}
					(AB)^{2}(A^{-1}B^{-1})^{2}&=(T(AB)AB-I)(T(AB)A^{-1}B^{-1}-I)\\
				&=T(AB)^{2}ABA^{-1}B^{-1}-T(AB)AB-T(AB)A^{-1}B^{-1}+I.
			\end{aligned}
		\end{equation}
		From (\ref{6.1}) and (\ref{6.2}) we can write, \\
		\begin{equation*}
			\begin{aligned}
				T(	[A,B]AB[A,B](AB)^{-1})&= T(T(AB)^{2}ABA^{-1}B^{-1}-T(AB)AB-T(AB)A^{-1}B^{-1}+I)	\\ 
				&=T(AB)^{2}T(ABA^{-1}B^{-1})-2T(AB)^{2}+2. 	
			\end{aligned}
		\end{equation*}
		So the polynomial corresponding to the given word $w(a,b)=[[a,b],(ab)[a,b](ab)^{-1}]$  is,\begin{center}
			$p_{[[a,b],(ab)[a,b](ab)^{-1}] }=z^{2}k(x,y,z)-2z^{2}+2$.
		\end{center}
        Consider the system of equations,
		\begin{equation}\label{6.3}
			\begin{aligned}
				p_{[a,b]}&=0,\\ p_{(ab)[a,b](ab)^{-1}}&=0,\\ p_{[a,b](ab)[a,b](ab)^{-1}}&=0.
			\end{aligned}
		\end{equation}
        This is same as,
		\begin{center}
		    $k(x,y,z)=0,$\\ \hspace{2cm} $z^{2}k(x,y,z)-2z^{2}+2=0.$
		\end{center}
			Thus, $(1,0,1)$ is a solution to (\ref{6.3}) in $\tau$.Therefore, the word map corresponding to the given word is surjective.
	\end{proof}

	\begin{prop}\label{a6}
		The word map induced by the word $w(a,b)=[[a,b],(ab)^{n}[a,b](ab)^{-n}]$, $n\geq 2$ is surjective on $$SU(2)$$.
	\end{prop}
	\begin{proof}
	Let $A,B \in SU(2)$. 	With the help of Lemma \ref{L1} we get, 
		\begin{equation*}
			\begin{aligned}
				[A,B](AB)^{n}[A,B](AB)^{-n}&= [A,B](U_{n-1}(z/2)AB-U_{n-2}(z/2)I)[A,B](U_{n-1}(z/2)(AB)^{-1}\\& \hspace*{1cm}-U_{n-2}(z/2)I)\\ &= U_{n-1}(z/2)^{2}[A,B]AB[A,B](AB)^{-1}-U_{n-1}(z/2)U_{n-2}(z/2)[A,B]AB[A,B]\\&\hspace*{0.5cm}-U_{n-1}(z/2)U_{n-2}(z/2)[A,B][A,B](AB)^{-1}+ U_{n-2}(z/2)^{2}[A,B]^{2}.
			\end{aligned}
		\end{equation*}
		Hence we get,
		\begin{equation}\label{7.1}
			\begin{aligned}
				T([A,B](AB)^{n}[A,B](AB)^{-n})&=U_{n-1}(z/2)^{2}T([A,B]AB[A,B](AB)^{-1})\\&-U_{n-1}(z/2)U_{n-2}(z/2)T([A,B]AB[A,B])\\ & -U_{n-1}(z/2)U_{n-2}(z/2)T([A,B][A,B](AB)^{-1})+ U_{n-2}(z/2)^{2}T([A,B]^{2}).
			\end{aligned}
		\end{equation}
		And
		\begin{equation}\label{7.2}
			\begin{aligned}
				T([A,B]AB[A,B](AB)^{-1})&=
				T(AB)^{2}T(ABA^{-1}B^{-1})-2T(AB)^{2}+2.
			\end{aligned}
		\end{equation}
		Also,  
		\begin{equation}\label{7.3}
			\begin{aligned}
				T([A,B]AB[A,B])&= T((T([A,B])[A,B]-I)AB) 
                \\&= T([A,B])T((AB)^{2}A^{-1}B^{-1})-T(AB)\\&= T([A,B])(T(AB)T(ABA^{-1}B^{-1})-T(AB))-T(AB)\\ &=
				T([A,B])^{2}T(AB)-T([A,B])T(AB))-T(AB).
			\end{aligned}
		\end{equation}
		\begin{equation}\label{7.4}
			\begin{aligned}
				T([A,B][A,B](AB)^{-1})&= T((T([A,B])[A,B]-I)(AB)^{-1})\\&=T([A,B])T([A,B](AB)^{-1})-T(AB)\\&= T([A,B])T(AB)-T(AB).
			\end{aligned}
		\end{equation}
		Substituting (\ref{7.2}),(\ref{7.3}) and (\ref{7.4}) in (\ref{7.1}), we get;
		\begin{equation}
			\begin{aligned}
				T([A,B](AB)^{n}[A,B](AB)^{-n})&=U_{n-1}(z/2)^{2}	T(AB)^{2}T(ABA^{-1}B^{-1})-2T(AB)^{2}+2\\&-U_{n-1}(z/2)U_{n-2}(z/2)	T([A,B])^{2}T(AB)-T([A,B])T(AB))-T(AB)\\ & -U_{n-1}(z/2)U_{n-2}(z/2)T([A,B])T(AB)-T(AB)+ U_{n-2}(z/2)^{2}(T([A,B])^{2}-2).
			\end{aligned}
		\end{equation}
		Now,
		\begin{equation*}
			\begin{aligned}
				p_{[a,b](ab)^{n}[a,b](ab)^{-n}}&=U_{n-1}(z/2)^{2}	z^{2}(k(x,y,z)-2z^{2}+2)\\&   -U_{n-1}(z/2)U_{n-2}(z/2)	(k(x,y,z)^{2}z-k(x,y,z)z-z)\\ & -U_{n-1}(z/2)U_{n-2}(z/2)(k(x,y,z)z-z)+ U_{n-2}(z/2)^{2}(k(x,y,z)^{2}-2),\\
				p_{[a,b]}&=k(x,y,z),\\
				p_{(ab)^{n}[a,b](ab)^{-n}}&= k(x,y,z).
			\end{aligned}
		\end{equation*}
        Consider the system of equations,
		\begin{equation}\label{7.b}
			\begin{aligned}
				p_{[a,b]}&=0,\\ p_{(ab)^{n}[a,b](ab)^{-n}}&=0,\\ p_{[a,b](ab)^{n}[a,b](ab)^{-n}}&=0.
			\end{aligned}
		\end{equation}
		The system (\ref{7.b}) becomes:
		\begin{center}
			$k(x,y,z)=0,$\\
            $U_{n-1}(z/2)^{2}	(z^{2}k(x,y,z)-2z^{2}+2)-U_{n-1}(z/2)U_{n-2}(z/2)	(k(x,y,z)^{2}z-k(x,y,z)z-z)$ \\ $-U_{n-1}(z/2)U_{n-2}(z/2)(k(x,y,z)z-z)+ U_{n-2}(z/2)^{2}(k(x,y,z)^{2}-2)=0. $
	\end{center}
	This is equivalent to,	
		$$k(x,y,z)=0,$$ $$	z^{2}U_{n-1}(z/2)^{2}-2U_{n-1}(z/2)^{2}+U_{n-1}(z/2)^{2}+U_{n-2}(z/2)^{2}-zU_{n-1}(z/2)U_{n-2}(z/2)=0.\\ $$
		Using the relation, $U_{n}(z)^{2}+U_{n-1}(z)^{2}-2zU_{n}(z)U_{n-1}(z)=1$ second equation becomes,
        \begin{center}
			$z^{2}U_{n-1}(z/2)^{2}-2U_{n-1}(z/2)^{2}+1=0 $
			$\Rightarrow$ $1+(z^{2}-2)U_{n-1}(z/2)^{2}=0.$
		\end{center}
So the system (\ref{7.b}) is equivalent to the following system.
		\begin{equation}\label{3}
			\begin{aligned}
				k(x,y,z)&=0,
				\\ 1+(z^{2}-2)U_{n-1}(z/2)^{2}&=0.
			\end{aligned}
		\end{equation}
		The given system of equations corresponds to that in Proposition \ref{a2}, with the variable $x$ replaced by $z$. Hence, we can find the solution in the same way as shown in that proposition by substituting $x$ with $z$. Therefore, the given word map is surjective.
	\end{proof}
	
	\newpage
	\begin{prop}\label{a7}
		The word map induced by the word $w(a,b)=[[a,b],a^{n}b[a,b]b^{-1}a^{-n}]$ is surjective for all $n\geq 2$ on $SU(2)$.
	\end{prop}
	
	\begin{proof}
    Let $A,B \in SU(2)$. By Lemma \ref{L1} we get
		\begin{equation*}
			\begin{aligned}
				[[A,B] A^{n}B[A,B]B^{-1}A^{-n}]&= [A,B](U_{n-1}(x/2)A-U_{n-2}(x/2)I)B[A,B]B^{-1}(U_{n-1}(x/2)A^{-1}\\&\hspace*{1cm} -U_{n-2}(x/2)I)\\ &= U_{n-1}(x/2)^{2}[A,B]AB[A,B]B^{-1}A^{-1}\\&\hspace*{.75cm}-U_{n-1}(x/2)U_{n-2}(x/2)[A,B]AB[A,B]B^{-1}\\&\hspace*{.75cm}-U_{n-1}(x/2)U_{n-2}(x/2)[A,B]B[A,B]B^{-1}A^{-1}\\&\hspace*{.75cm}+ U_{n-2}(x/2)^{2}[A,B]B[A,B]B^{-1}.
			\end{aligned}
		\end{equation*}
		Therefore,
		\begin{equation}\label{8.5}
			\begin{aligned}
				T([[A,B] A^{n}B[A,B]B^{-1}A^{-n}])&=  U_{n-1}(x/2)^{2}T([A,B]AB[A,B]B^{-1}A^{-1})\\&\hspace*{.75cm}-U_{n-1}(x/2)U_{n-2}(x/2)T([A,B]AB[A,B]B^{-1})\\&\hspace*{.75cm}-U_{n-1}(x/2)U_{n-2}(x/2)T([A,B]B[A,B]B^{-1}A^{-1})\\&\hspace*{.75cm}+ U_{n-2}(x/2)^{2}T([A,B]B[A,B]B^{-1}).
			\end{aligned}
		\end{equation}
      And
		\begin{equation}\label{8.1}
			\begin{aligned}
T([A,B]AB[A,B]B^{-1}A^{-1})&=
                T(AB)^{2}T([A,B])-2T(AB)^{2}+2,
			\end{aligned}
		\end{equation}
		\begin{equation}\label{8.2}
			\begin{aligned}
				T([A,B]AB[A,B]B^{-1})&=
                T(T(B)[A,B]AB[A,B]-[A,B]AB[A,B]B)\\&=T(B)T([A,B]^{2}AB)-T([A,B]AB[A,B]B)\\&= T(B)T([A,B])^{2}T(AB)-2T(B)T([A,B])T(AB)+T(A),
			\end{aligned}
		\end{equation}
		\begin{equation}\label{8.3}
			\begin{aligned}
            T([A,B]B[A,B]B^{-1}A^{-1})&=T(ABA^{-1}B^{-1}BABA^{-1}B^{-1}B^{-1}A^{-1})=T(A),
			\end{aligned}
		\end{equation}
		\begin{equation}\label{8.4}
			\begin{aligned}
T([A,B]B[A,B]B^{-1})&=T(ABA^{-1}B^{-1}BABA^{-1}B^{-1}B^{-1})\\&=T(B)^{2}T([A,B])-2T(B)^{2}+2.
			\end{aligned}
		\end{equation}
		Substituting (\ref{8.2}),(\ref{8.3}),(\ref{8.4}) and (\ref{8.1}) in (\ref{8.5}) , we get;
		\begin{equation}
			\begin{aligned}
				T([[A,B] A^{n}B[A,B]B^{-1}A^{-n}])&=  U_{n-1}(x/2)^{2}(T(AB)^{2}T([A,B])-2T(AB)^{2}+2)\\& -U_{n-1}(x/2)U_{n-2}(x/2)(T(B)T([A,B])^{2}T(AB))\\&-U_{n-1}(x/2)U_{n-2}(x/2)(-2T(B)T([A,B])T(AB)+T(A))\\&-U_{n-1}(x/2)U_{n-2}(x/2)T(A)\\&+ U_{n-2}(x/2)^{2}(T(B)^{2}T([A,B])-2T(B)^{2}+2).
			\end{aligned}
		\end{equation}
		Now,
		\begin{equation*}
			\begin{aligned}
				p_{[a,b]a^{n}b[a,b]b^{-1}a^{-n}}&=U_{n-1}(x/2)^{2}(z^{2}k(x,y,z)-2z^{2}+2)\\&\hspace*{0.5cm}-U_{n-1}(x/2)U_{n-2}(x/2)(yzk(x,y,z)^{2}-2yzk(x,y,z)+x)\\& \hspace*{0.5cm} -U_{n-1}(x/2)U_{n-2}(x/2)x+ U_{n-2}(x/2)^{2}(k(x,y,z)y^{2}-2y^{2}+2),\\
				p_{[a,b]}&=k(x,y,z),\\
				p_{a^{n}b[a,b]b^{-1}a^{-n}}&= k(x,y,z).
			\end{aligned}
		\end{equation*}
        Consider the system of equations,
		\begin{equation}\label{8}
			\begin{aligned}
				p_{[a,b]}&=0,\\ p_{a^{n}b[a,b]b^{-1}a^{-n}}&=0,\\ p_{[a,b]a^{n}b[a,b]b^{-1}a^{-n}}&=0.
			\end{aligned}
		\end{equation}
		This is same as,
		\begin{center}
			$k(x,y,z)=0,$\\
            $U_{n-1}(x/2)^{2}(-2z^{2}+2)-U_{n-1}(x/2)U_{n-2}(x/2)(2x)+ U_{n-2}(x/2)^{2}(-2y^{2}+2)=0 .$
		\end{center}
		
 Using the relation $U_{n}(x)^{2}+U_{n-1}(x)^{2}-2xU_{n}(x)U_{n-1}(x)=1$ the system (\ref{8}) is equivalent to,	
\begin{center}
$k(x,y,z)=0,$\\
			$2(1-y^{2} U_{n-2}(x/2)^{2}-z^{2}U_{n-1}(x/2)^{2})=0$.
	\end{center}
		We can see that $(0,1,1)\in \tau$ is a solution to this system. So the given word map is surjective.
	\end{proof}

		\begin{prop}\label{a8}
		The word map induced by the word $w(a,b)=[[a,b],ab^{2}[a,b]b^{-2}a]$, $n\geq 2$ is surjective on $SU(2)$.
	\end{prop}
	\begin{proof}
    
	Let $A,B\in SU(2)$. Consider $T([A,B])=T(ABA^{-1}B^{-1})$, $T(AB^{2}[A,B]B^{-2}A^{-1})=T(ABA^{-1}B^{-1})$.
	\begin{equation*}\label{}
		\begin{aligned}
			T([[A,B] AB^{2}[A,B]B^{-2}A^{-1}])&=  T(B)^{2}T([A,B]AB[A,B]B^{-1}A^{-1})-T(B)T([A,B]AB[A,B]A^{-1})\\& \hspace*{0.75cm}-T(B)T([A,B]A[A,B]B^{-1}A^{-1})+T([A,B]A[A,B]A^{-1})\\ & = T(B)^{2}( T(AB)^{2}T( A^{-1}B^{-1}AB)-T(AB)T(A^{-1}B^{-1})-T(AB)^{2}+2) \\&\hspace*{.75cm}-T(B)(T(A)T(AB)T(A^{-1}B^{-1}AB)-2T(A)T(AB)+T(B))\\&\hspace*{.75cm}-T(B)(T(A)T(AB)T(A^{-1}B^{-1}AB)-2T(AB)T(A)+T(B))\\&\hspace*{.75cm}+T(A)^{2}T(A^{-1}B^{-1}AB)-2T(A)T(A)+2 
			\\ &=  T(B)^{2}( T(AB)^{2}T( A^{-1}B^{-1}AB)-2T(AB)^{2}+2) \\&\hspace*{.75cm}-2T(B)(T(A)T(AB)T(A^{-1}B^{-1}AB)-2T(A)T(AB)+T(B))\\&\hspace*{.75cm}+T(A)^{2}T(A^{-1}B^{-1}AB)-2T(A)^{2}+2.
		\end{aligned}
	\end{equation*}
    Therefore the polynomials corresponding to the words are;
\begin{equation*}
	\begin{aligned}
		p_{[a,b]ab^{2}[a,b]b^{-2}a^{-1}}&=y^{2}(z^{2}k(x,y,z)-2z^{2}+2)-2y(xzk(x,y,z)-2xz+y)+x^{2}k(x,y,z)-2x^{2}+2,\\
		p_{[a,b]}&=k(x,y,z),\\
		p_{ab^{2}[a,b]b^{-2}a^{-1}}&= k(x,y,z).
	\end{aligned}
\end{equation*}
Consider the system of equations,
\begin{equation}\label{34}
	\begin{aligned}
		p_{[a,b]}&=0,\\ p_{ab^{2}[a,b]b^{-2}a^{-1}}&=0,\\ p_{[a,b]ab^{2}[a,b]b^{-2}a^{-1}}&=0.
	\end{aligned}
\end{equation}
The above system of equations is equivalent to:
\begin{center}
	$k(x,y,z)=0,$\\$y^{2}(-2z^{2}+2)-2y(-2xz+y)-2x^{2}+2=0.$
\end{center}
Clearly, $(1,1,0)\in \tau $ is a solution to system (\ref{34}) of equations. Thus, the word map corresponding to the given word is surjective. 
\end{proof}
	\begin{prop}\label{a9}
		The word map induced by the word $w(a,b)=[[a,b],a^{n}b^{2}[a,b]b^{-2}a^{-n}]$, $n\geq 2$ is surjective on $SU(2)$.
	\end{prop}
\begin{proof}
Let $A$ and $B$ be matrices in $SU(2)$. Using (1) of Lemma \ref{L1},
		\begin{equation*}
		\begin{aligned}
			[[A,B] A^{n}B^{2}[A,B]B^{-2}A^{-n}]&= [A,B](U_{n-1}(x/2)A-U_{n-2}(x/2)I)B^{2}[A,B]B^{-2}(U_{n-1}(x/2)A^{-1}\\&\hspace*{1cm} -U_{n-2}(x/2)I)\\ &= U_{n-1}(x/2)^{2}[A,B]AB^{2}[A,B]B^{-2}A^{-1}\\&\hspace*{.75cm}-U_{n-1}(x/2)U_{n-2}(x/2)[A,B]AB^{2}[A,B]B^{-2}\\&\hspace*{.75cm}-U_{n-1}(x/2)U_{n-2}(x/2)[A,B]B^{2}[A,B]B^{-2}A^{-1}\\&\hspace*{.75cm}+ U_{n-2}(x/2)^{2}[A,B]B^{2}[A,B]B^{-2}.
		\end{aligned}
	\end{equation*}
	Therefore, \begin{equation}\label{9.5}
		\begin{aligned}
			T([[A,B] A^{n}B^{2}[A,B]B^{-2}A^{-n}])&=  U_{n-1}(x/2)^{2}T([A,B]AB^{2}[A,B]B^{-2}A^{-1})\\&\hspace*{.5cm}-U_{n-1}(x/2)U_{n-2}(x/2)T([A,B]AB^{2}[A,B]B^{-2})\\&\hspace*{.5cm}-U_{n-1}(x/2)U_{n-2}(x/2)T([A,B]B^{2}[A,B]B^{-2}A^{-1})\\&\hspace*{.5cm}+ U_{n-2}(x/2)^{2}T([A,B]B^{2}[A,B]B^{-2}).
		\end{aligned}
	\end{equation}
    Also,
	\begin{equation}\label{9.1}
		\begin{aligned}
			T([A,B]AB^{2}[A,B]B^{-2}A^{-1})&=T([A,B]A(T(B)B-I)[A,B](T(B)B^{-1}-I)A^{-1})\\&= T(B)^{2}T([A,B]AB[A,B]B^{-1}A^{-1})-T(B)T([A,B]AB[A,B]A^{-1})\\&\hspace{0.5cm}-T(B)T([A,B]A[A,B]B^{-1}A^{-1})+T([A,B]A[A,B]A^{-1})\\&= T(B)^{2}(T(AB)^{2}T(ABA^{-1}B^{-1})-2T(AB)^{2}+2)\\&-2T(B)(T(A)T(AB)T(ABA^{-1}B^{-1})-2T(A)T(AB)+T(B))\\&+T(A)^{2}T(ABA^{-1}B^{-1})-2T(A)^{2}+2,
		\end{aligned}
	\end{equation}
	\begin{equation}\label{9.2}
		\begin{aligned}
			T([A,B]AB^{2}[A,B]B^{-2})&=T([A,B]A(T(B)B-I)[A,B](T(B)B^{-1}-I))\\&= T(B)^{2}T([A,B]AB[A,B]B^{-1})-T(B)T([A,B]AB[A,B])\\&\hspace*{0.5cm}-T(B)T([A,B]A[A,B]B^{-1})+T([A,B]A[A,B])\\&= T(ABA^{-1}B^{-1})^{2}T(B)^{3}T(AB)-2T(B)^{2}T(AB)T(ABA^{-1}B^{-1})\\&\hspace*{0.5cm}-T(ABA^{-1}B^{-1})^{2}T(B)T(AB)+T(ABA^{-1}B^{-1})T(B)T(AB)\\&\hspace*{0.5cm}-T(A)T(B)^{2}T(ABA^{-1}B^{-1})^{2}+2T(A)T(B)^{2}T(ABA^{-1}B^{-1})\\&\hspace*{0.5cm}+T(A)T(ABA^{-1}B^{-1})^{2}-T(A)T(ABA^{-1}B^{-1})+2T(B)T(AB)-T(A),
		\end{aligned}
	\end{equation}
	\begin{equation}\label{9.3}
		\begin{aligned}
			T([A,B]B^{2}[A,B]B^{-2}A^{-1})&=T(B)^{2}T([A,B]B[A,B]B^{-1}A^{-1}-T(B)T([A,B]B[A,B]A^{-1})\\&\hspace*{0.5cm}-T(B)T([A,B][A,B]B^{-1}A^{-1})+T([A,B]^{2}A^{-1})\\&= 2T(B)T(AB)-T(B)T(AB)T(ABA^{-1}B^{-1})\\&\hspace*{0.5cm}+T(ABA^{-1}B^{-1})T(A)-T(A),
		\end{aligned}
	\end{equation}
	\begin{equation}\label{9.4}
		\begin{aligned}
			T([A,B]B^{2}[A,B]B^{-2})&=T([A,B](T(B)B-I)[A,B](T(B)B-I))\\&=
			T(B)^{2}T([A,B]B[A,B]B^{-1})-T(B)T([A,B]B[A,B])\\&\hspace*{0.5cm}-T(B)T([A,B]^{2}B^{-1})+T([A,B]^{2})\\&= T(B)^{4}T(ABA^{-1}B^{-1})-2T(B)^{4}+4T(B)^{2}-T(ABA^{-1}B^{-1})^{2}T(B)^{2}\\&\hspace*{0.5cm}+T(ABA^{-1}B^{-1})^{2}-2.
		\end{aligned}
	\end{equation}
	Substituting (\ref{9.2}),(\ref{9.3}),(\ref{9.4}) and (\ref{9.1}) in (\ref{9.5}) , we get;
	\begin{equation}
		\begin{aligned}
			T([[A,&B] A^{n}B^{2}[A,B]B^{-2}A^{-n}])=  U_{n-1}(x/2)^{2}(T(B)^{2}(T(AB)^{2}T(ABA^{-1}B^{-1})-2T(AB)^{2}+2)\\&-2T(B)(T(A)T(AB)T(ABA^{-1}B^{-1})-2T(A)T(AB)+T(B))+T(A)^{2}T(ABA^{-1}B^{-1})\\&-2T(A)^{2}+2)-U_{n-1}(x/2)U_{n-2}(x/2)(T(ABA^{-1}B^{-1})^{2}T(B)^{3}T(AB)\\&-2T(B)^{2}T(AB)T(ABA^{-1}B^{-1})-T(ABA^{-1}B^{-1})^{2}T(B)T(AB)+T(ABA^{-1}B^{-1})T(B)T(AB)\\&-T(A)T(B)^{2}T(ABA^{-1}B^{-1})^{2}+2T(A)T(B)^{2}T(ABA^{-1}B^{-1})+T(A)T(ABA^{-1}B^{-1})^{2}\\&-T(A)T(ABA^{-1}B^{-1})+2T(B)T(AB)-T(A))-U_{n-1}(x/2)U_{n-2}(x/2)(2T(B)T(AB)\\&-T(B)T(AB)T(ABA^{-1}B^{-1})+ U_{n-2}(x/2)^{2}(T(B)^{4}T(ABA^{-1}B^{-1})-2T(B)^{4}+4T(B)^{2}\\&-T(ABA^{-1}B^{-1})^{2}T(B)^{2}+T(ABA^{-1}B^{-1})^{2}-2).
		\end{aligned}
	\end{equation}
	The polynomials corresponding to the words are \begin{equation*}
		\begin{aligned}
			p_{[a,b]a^{n}b^{2}[a,b]b^{-2}a^{-n}}&=U_{n-1}(x/2)^{2}(y^{2}z^{2}k(x,y,z)-2y^{2}z^{2}-2xyzk(x,y,z)+4xyz+x^{2}k(x,y,z)-2x^{2}+2)\\&\hspace*{0.5cm}-U_{n-1}(x/2)U_{n-2}(x/2)(y^{3}zk(x,y,z)^{2}-2y^{2}zk(x,y,z)-yzk(x,y,z)^{2}+yzk\\&\hspace*{0.5cm}-xy^{2}k(x,y,z)^{2}+2xy^{2}k(x,y,z)+xk(x,y,z)^{2}-xk(x,y,z)+2yz-x)\\&\hspace*{0.5cm} -U_{n-1}(x/2)U_{n-2}(x/2)(2yz-yzk(x,y,z)+xk(x,y,z)-x)\\&\hspace*{0.5cm}+ U_{n-2}(x/2)^{2}(k(x,y,z)y^{4}-2y^{4}+4y^{2}-k(x,y,z)^{2}y^{2}+k(x,y,z)^{2}-2),\\
			p_{[a,b]}&=k(x,y,z),\\
			p_{a^{n}b^{2}[a,b]b^{-2}a^{-n}}&= k(x,y,z).
		\end{aligned}
	\end{equation*}
    Consider the system of equations,
	\begin{equation}\label{7}
		\begin{aligned}
			p_{[a,b]}&=0,\\ p_{a^{n}b^{2}[a,b]b^{-2}a^{-n}}&=0,\\ p_{[a,b]a^{n}b^{2}[a,b]b^{-2}a^{-n}}&=0.
		\end{aligned}
	\end{equation}
	The system (\ref{7}) is equivalent to:  \begin{center}
		$k(x,y,z)=0,$
        $U_{n-1}(x/2)^{2}(4xyz-2y^{2}z^{2}-2x^{2}+2)-U_{n-1}(x/2)U_{n-2}(x/2)(4yz-2x)+ U_{n-2}(x/2)^{2}(4y^{2}-2y^{4}-2)=0 .$
	\end{center}
	When $n\equiv 0,1,3,4$ mod(6), $(1,1,0)\in \tau$ is a solution to (\ref{7}).
When $n\equiv 2,5$ mod(6), $(1,0,1)\in \tau$ is solution to (\ref{7}). 
Thus we say that the word map corresponding to the given word is surjective.
\end{proof}

\newpage
\begin{prop}\label{a10}
		The word map induced by the word $w(a,b)=[[a,b],a^{n}b^{m}[a,b]b^{-m}a^{-n}]$, $n\geq 2$ is surjective on $SU(2)$ when  $n\equiv 1,2,4,5 $ (mod $6$) and $m\equiv 0,2$ (mod $4$). 
	\end{prop}
	\begin{proof}
	Let $A,B\in SU(2)$. We have $T([A,B])=T(ABA^{-1}B^{-1})$, $T(A^{n}B^{m}[A,B]B^{-m}A^{-n})=T(ABA^{-1}B^{-1})$.\\
    Using (1) of Lemma \ref{L1},
		\begin{equation*}\label{}
			\begin{aligned}
				[[A,B] A^{n}B^{m}[A,B]B^{-m}A^{-n}]&= [A,B](U_{n-1}(x/2)A-U_{n-2}(x/2)I)(U_{m-1}(y/2)B-U_{m-2}(y/2)I)
				\\& \hspace*{0.5cm} [A,B](U_{m-1}(y/2)B^{-1}-U_{m-2}(y/2)I)(U_{n-1}(x/2)A^{-1}-U_{n-2}(x/2)I)
				\\&= (U_{n-1}(x/2)U_{m-1}(y/2)[A,B]AB-U_{n-1}(x/2)U_{m-2}(y/2)[A,B]A\\& \hspace*{0.75cm}+U_{n-2}(x/2)U_{m-2}(y/2)[A,B]-U_{n-2}(x/2)U_{m-1}(y/2)[A,B]B)\\& \hspace*{0.2cm}
				(U_{n-1}(x/2)U_{m-1}(y/2)[A,B]B^{-1}A^{-1}-U_{m-1}(y/2)U_{n-2}(x/2)[A,B]B^{-1}\\& \hspace*{0.75cm}-U_{m-2}(y/2)U_{n-1}(x/2)[A,B]A^{-1}+U_{m-2}(y/2)U_{n-2}(x/2)[A,B])
			\end{aligned}
		\end{equation*}
			Therefore,
            \begin{equation*}\label{}
			\begin{aligned}
				T([[A,B] A^{n}B^{m}[A,B]B^{-m}A^{-n}])&= U_{n-1}(x/2)^{2}U_{m-1}(y/2)^{2}(k(x,y,z)z^{2}-2z^{2}+2)\\& \hspace*{-1.2cm}-U_{n-1}(x/2)U_{n-2}(x/2)U_{m-1}(y/2)^{2}(k(x,y,z)^{2}yz-2yz+x)\\& \hspace*{-1.2cm}-2U_{n-1}(x/2)^{2}U_{m-1}(y/2)U_{m-2}(y/2)^{2}(k(x,y,z)xz-4xz+2y)\\& \hspace*{-5.2cm}+U_{n-1}(x/2)U_{n-2}(x/2)U_{m-1}(y/2)U_{m-2}(y/2)(k(x,y,z)^{2}z-2z+(x^{2}y-2xy)k(x,y,z)-x^{2}y+2xy)\\& \hspace*{0.2cm}+U_{n-1}(x/2)^{2}U_{m-2}(y/2)^{2}(x^{2}k(x,y,z)-2x^{2}+2)\\& \hspace*{0.2cm}-U_{n-1}(x/2)U_{n-2}(x/2)U_{m-2}(y/2)^{2}(k(x,y,z)^{2}x-2x)\\& \hspace*{0.2cm}-U_{n-2}(x/2)^{2}U_{m-1}(y/2)U_{m-2}(y/2)(k(x,y,z)^{2}y-yk(x,y,z)-y)\\& \hspace*{0.2cm}-U_{m-2}(y/2)^{2}U_{n-1}(x/2)U_{n-2}(x/2)(k(x,y,z)x-x)\\& \hspace*{0.2cm}+U_{n-2}(x/2)^{2}U_{m-2}(y/2)^{2}(k(x,y,z)^{2}-2)\\& \hspace*{0.2cm}-U_{m-1}(y/2)^{2}U_{n-1}(x/2)U_{n-2}(x/2)x\\& \hspace*{0.2cm}+U_{n-2}(x/2)^{2}U_{m-1}(y/2)^{2}(y^{2}k(x,y,z)-2y^{2}+2)\\& \hspace*{0.2cm}-U_{n-2}(x/2)^{2}U_{m-1}(y/2)U_{m-2}(y/2)(k(x,y,z)y-y)
			\end{aligned}
		\end{equation*}
		Consider the system of equations,
			\begin{equation}\label{}
			\begin{aligned}
				p_{[a,b]}&=0\\ p_{a^{n}b^{m}[a,b]b^{-m}a^{-n}}&=0\\ p_{[a,b]a^{n}b^{m}[a,b]b^{-m}a^{-n}}&=0
			\end{aligned}
		\end{equation}
		The above system of equations is
			\begin{equation*}
			\begin{aligned}
            p_{[a,b]}&=k(x,y,z)=0\\
				p_{a^{n}b^{m}[a,b]b^{-m}a^{-n}}&= k(x,y,z)=0\\
				p_{[a,b]a^{n}b^{m}[a,b]b^{-m}a^{-n}}&=2z^{2}U_{n-1}(x/2)^{2}U_{m-1}(y/2)^{2}-2yU_{n-1}(x/2)^{2}U_{m-1}(y/2)U_{m-2}(y/2)\\&\hspace*{-1.95cm}+2U_{m-1}(y/2)^{2}+4xzU_{n-1}(x/2)^{2}U_{m-1}(y/2)^{2}U_{m-2}(y/2)+(2-2x^{2})U_{n-1}(x/2)^{2}U_{m-2}(y/2)^{2}\\&\hspace*{-1.95cm}+(2xy-4z)U_{n-1}(x/2)U_{m-1}(y/2)U_{n-2}(x/2)U_{m-2}(y/2)+2xU_{n-1}(x/2)U_{n-2}(x/2)U_{m-2}(y/2)^{2}\\&\hspace*{-1.95cm}+2yU_{n-2}(x/2)^{2}U_{m-1}(y/2)U_{m-2}(y/2)-2U_{n-2}(x/2)^{2}U_{m-2}(y/2)^{2}-2y^{2}U_{m-1}(y/2)^{2}U_{n-2}(x/2)^{2}=0	
			\end{aligned}
		\end{equation*}

		Clearly, $(1,0,1)\in \tau $ is a solution to the above system of equations when $n\equiv 1,2,4,5 $ (mod $6$) and $m\equiv 0,2$ (mod $4$). 
		\end{proof}

	\begin{prop}\label{a11}
		The word map induced by the word $w(a,b)=[[a,b],[a,(ab)^{n}]]$, $n\geq 2$ is surjective on $SU(2)$.
	\end{prop}
	\begin{proof}
		Let $A,B\in Su(2)$. Consider,
		\begin{equation*}
			\begin{aligned}
				[A,(AB)^{n}]&= A(AB)^{n}A^{-1}(AB)^{-n}\\ &=A(U_{n-1}(z/2)AB-U_{n-2}(z/2)I)A^{-1}(U_{n-1}(z/2)B^{-1}A^{-1}-U_{n-2}(z/2)I)\\&= U_{n-1}(z/2)^{2}A^{2}BA^{-1}B^{-1}A^{-1}-U_{n-1}(z/2)U_{n-2}(z/2)A^{2}BA^{-1}\\&\hspace*{0.5cm}-U_{n-1}(z/2)U_{n-2}(z/2)B^{-1}A^{-1}+ U_{n-2}(z/2)^{2}I
			\end{aligned}
		\end{equation*} 
		\begin{equation*}
			\begin{aligned}
				[A,B][A,{(AB)}^{n}]&= [A,B][A,(U_{n-1}(z/2)AB-U_{n-2}(z/2)I)]\\ &=ABA^{-1}B^{-1}A(U_{n-1}(z/2)AB-U_{n-2}(z/2)I)A^{-1}(U_{n-1}(z/2)B^{-1}A^{-1}-U_{n-2}(z/2)I)\\&= U_{n-1}(z/2)^{2}ABA^{-1}B^{-1}A^{2}BA^{-1}B^{-1}A^{-1}-U_{n-1}(z/2)U_{n-2}(z/2)ABA^{-1}B^{-1}A^{2}BA^{-1}\\&\hspace*{0.5cm}-U_{n-1}(z/2)U_{n-2}(z/2)ABA^{-1}B^{-1}AA^{-1}B^{-1}A^{-1}+ U_{n-2}(z/2)^{2}ABA^{-1}B^{-1}AA^{-1}.
			\end{aligned}
		\end{equation*}
		Hence we get,
		\begin{equation}
			\begin{aligned}
				T([A,{(AB)}^{n}])&=U_{n-1}(z/2)^{2}T(ABA^{-1}B^{-1})-2U_{n-1}(z/2)U_{n-2}(z/2)T(AB)+ 2U_{n-2}(z/2)^{2},
			\end{aligned}
		\end{equation}
		\begin{equation}\label{11.1}
			\begin{aligned}
				T([A,B][A,{(AB)}^{n}])&=U_{n-1}(z/2)^{2}T(A^{-1}B^{-1}A^{2}BA^{-1})-U_{n-1}(z/2)U_{n-2}(z/2)T(B^{2}A^{-1}B^{-1}A^{2})\\&\hspace*{0.5cm}-U_{n-1}(z/2)U_{n-2}(z/2)T(A^{-1}B^{-1})+ U_{n-2}(z/2)^{2}T(ABA^{-1}B^{-1}).
			\end{aligned}
		\end{equation}
		
		And
		\begin{equation}\label{11.2}
			\begin{aligned}
				T(A^{-2}B^{-1}A^{2}B)&=T((T(A)A^{-1}-I)B^{-1}(T(A)A-I)B)\\&=T(T(A)^{2}A^{-1}B^{-1}AB-T(A)A^{-1}-T(A)B^{-1}AB+I)\\&=T(A)^{2}T(A^{-1}B^{-1}AB)-2T(A)^{2}+2.
			\end{aligned}
		\end{equation}
		
		Also,
		\begin{equation}\label{11.3}
			\begin{aligned}
				T(B^{2}A^{-1}B^{-1}A^{2})&= T((T(B)B-I)A^{-1}B^{-1}(T(A)A-I))\\&=T(T(A)T(B)BA^{-1}B^{-1}A-T(B)BA^{-1}B^{-1}-T(A)A^{-1}B^{-1}A+A^{-1}B^{-1})\\&= T(A)T(B)T(ABA^{-1}B^{-1})-2T(A)T(B)+T(AB).
			\end{aligned}
		\end{equation}

		Substituting (\ref{11.2}) and (\ref{11.3}) in (\ref{11.1}), we get;
		\begin{equation}
			\begin{aligned}
				T([A,B][A,{(AB)}^{n}])&=U_{n-1}(z/2)^{2}(T(A)^{2}T(A^{-1}B^{-1}AB)-2T(A)^{2}+2)\\&\hspace*{0.5cm}-U_{n-1}(z/2)U_{n-2}(z/2)(T(A)T(B)T(ABA^{-1}B^{-1})-2T(A)T(B)+T(AB))\\&\hspace*{0.5cm}-U_{n-1}(z/2)U_{n-2}(z/2)T(A^{-1}B^{-1})+ U_{n-2}(z/2)^{2}T(ABA^{-1}B^{-1}).
			\end{aligned}
		\end{equation}
		So the polynomials are,
		\begin{equation*}
			\begin{aligned}
				p_{[a,b][a,(ab)^{n}]}&=U_{n-1}(z/2)^{2}(x^{2}k(x,y,z)-2x^{2}+2)\\&\hspace*{0.5cm}-U_{n-1}(z/2)U_{n-2}(z/2)(xyk(x,y,z)-2xy+z)\\&\hspace*{0.5cm}-U_{n-1}(z/2)U_{n-2}(z/2)z+ U_{n-2}(z/2)^{2}k(x,y,z),\\
				p_{[a,b]}&=k(x,y,z),\\
				p_{[a,(ab)^{n}]}&= U_{n-1}(z/2)^{2}k(x,y,z)-2U_{n-1}(z/2)U_{n-2}(z/2)z+ 2U_{n-2}(z/2)^{2}.
			\end{aligned}
		\end{equation*}
        Consider the system,
		\begin{equation}\label{5}
			\begin{aligned}
				p_{[a,b]}&=0,\\ 	p_{[a,(ab)^{n}]}&=0, \\ 	p_{[a,b][a,(ab)^{n}]}&=0.
			\end{aligned}
		\end{equation}
		The system (\ref{5}) is equivalent to the following system,
		\begin{equation}\label{11.5}
			\begin{aligned}
				k(x,y,z)&=0,\\
				U_{n-2}(z/2)(U_{n-2}(z/2)-U_{n-1}(z/2)z) &=0 ,\\
				U_{n-1}(z/2)(U_{n-1}(z/2)(1-x^{2})+U_{n-2}(z/2)(xy-z))&=0.
			\end{aligned}		
		\end{equation}
		
		When $n=2$, the system (\ref{11.5}) becomes 
		\begin{equation}\label{5.6}
			\begin{aligned}
				k(x,y,z)&=0,\\
				1-z^{2} &=0, \\
				z(z(1-x^{2})+xy-z)&=0.
			\end{aligned}		
		\end{equation}
		Clearly, $(1,1,1)\in \tau $ is solution to the system (\ref{11.5}). Thus, when $n=2$ the given word becomes surjective. 
		
		When $n=4$,  the system (\ref{11.5})  becomes
		\begin{equation}\label{5.7}
			\begin{aligned}
				k(x,y,z)&=0,\\
				(z^{2}-1)((z^{2}-1)-(z^{3}-2z)z) &=0, \\
				(z^{3}-2z)((z^{3}-2z)(1-x^{2})+(z^{2}-1)(xy-z))&=0.
			\end{aligned}		
		\end{equation}
		In this case also $(1,1,1)\in \tau $ is a solution  the system (\ref{5.7}).\\
		When $n$ is an odd number then $U_{n-2}(0)=0$. So $(1,1,0)\in \tau $ is a solution to (\ref{5}). Hence, when $n$ is odd the given word map is surjective.
		
		Now consider the case when $n$ is even and greater than 4, i.e. $n>4$. Take $z=2\mathrm{cos}(\frac{(n-2)\pi }{2(n-1)})$.
		\begin{equation*}
			\begin{aligned}
				U_{n-2}\left(\frac{z}{2}\right)=	U_{n-2}\left(\mathrm{cos}(\frac{(n-2)\pi }{2(n-1)}) \right) =\frac{\mathrm{sin}(\frac{(n-1)(n-2)\pi }{2(n-1)})}{\mathrm{sin}(\frac{(n-2)\pi }{2(n-1)})}=0.
			\end{aligned}
		\end{equation*}
		Also,
		\begin{center}
       
			$\mathrm{cos}\left(\frac{(n-2)\pi}{2(n-1)} \right)<\frac{1}{2}=\mathrm{cos}\left(\frac{\pi}{3}\right)\Leftrightarrow  \frac{(n-2)\pi}{2(n-1)} > \frac{\pi}{3} \Leftrightarrow 3n-6 > 2n-2 \Leftrightarrow n>4.$
		\end{center}
		\begin{center}
			$\Rightarrow $ $z=2\mathrm{cos}(\frac{(n-2)\pi }{2(n-1)})< 1 $.
		\end{center}
		Take $x=1$. Then we can find a $y\in [-2,2]$ such that $(1,y,2\mathrm{cos}(\frac{(n-2)\pi }{2(n-1)}))\in \tau $ is solution to (\ref{11.5}).	
	\end{proof}
	
	\textit{Proof of Theorem \ref{A}}. (a) The Proposition \ref{a1} gives the case of $n=1$ and proposition \ref{a2} will give the rest. (b) Proposition \ref{a3} and proposition \ref{a4} will prove the complete case. (c) proposition \ref{a5} and \ref{a6} together completes the proof. (d) Propostion \ref{a7}. (e) Proposition \ref{a8} and proposition \ref{a9} complete the proof. (f)	Proposition \ref{a10} gives the  proof. (g)	Proposition \ref{a11} gives the  proof.
	
	\section{Surjectivity of word maps in $SL(2,\mathbb{C})$}
    In this section, we prove the surjectivity of certain word maps in $SL(2,\mathbb{C})$ with the help of SageMath computations. \\
	Define $C(A):= \{gAg^{-1}\mid g\in SL(2,\mathbb{C})\}$. Then it is easy to see the following:
	$$SL(2,\mathbb{C})=\{I\}\bigcup \{-I\}\bigcup C\left(
	\begin{bmatrix}
		-1 & 1 \\
		0 & -1
	\end{bmatrix}
	\right) \bigcup C\left(
	\begin{bmatrix}
		1 & 1 \\
		0 & 1
	\end{bmatrix}
	\right)\bigcup\limits_{t\in C^{*},t\neq \pm1} C\left(
	\begin{bmatrix}
		t & 0 \\
		0 & t^{-1}
	\end{bmatrix}
	\right)$$\\
	We know that, $w(PAP^{-1},PBP^{-1})=Pw(A,B)P^{-1}$.
	
	\begin{lemma}\label{onto}\cite{TBYGZ}
		A regular non-constant function on $SL(2,\mathbb{C})^{n}$ omits no values in $\mathbb{C}$.  
	\end{lemma}
	\begin{lemma}\label{1}\cite{TBYGZ}
		For every nontrivial word $w(x,y)$ the image contains every element $z \in SL(2,\mathbb{C})$ with $T(z)\neq \pm2$.
	\end{lemma}
	\begin{proof}
		Choose $\alpha \neq \pm 2$. Let $G=\{(a_1,b_1,c_1,d_1)\in \mathbb{C}^4|a_{1}d_{1}-b_{1}c_{1}=1\}\subset \mathbb{C}^{4}$. Note that $G\cong SL(2,\mathbb{C})$.
		Define a function $f: G\times G \to \mathbb{C}$ by  $$f((a_{1},b_{1},c_{1},d_{1}),(a_{2},b_{2},c_{2},d_{2}))=T(w(a,b)),$$ where $T$ is the trace and note that it is a polynomial with integer coefficients. Therefore, by the Lemma \ref{onto}, there are elements $a,b\in SL(2,\mathbb{C})$  such that $T(w(a,b))=\alpha$. Let $z\in SL(2,\mathbb{C})$ with $T(z)=\alpha$. Let $u=w(a,b)$. So, $T(u)=T(z)$.
		Thus, $z$ and $u$ are conjugates, i.e., there exist $v\in SL(2,\mathbb{C})$ such that $z=vuv^{-1}$. 
		$$vuv^{-1}=vw(a,b)v^{-1}=w(vav^{-1},vbv^{-1}).$$ Hence $z=w(vav^{-1},vbv^{-1})$. 
		So $z$ is in image of $w$ for all $z\in SL(2,\mathbb{C})$ with $T(z)=\alpha\neq \pm 2$.
	\end{proof}
	\begin{remark}\label{Remark1}
      ~
		\begin{enumerate}
			\item From the proof of Lemma \ref{1}, it is clear that if $z\in \text{Im }w$, then all elements conjugate to $z$ is also there in image of $w$. In particular, $\text{Im }w$ is invariant under action by conjugation.
			\item Thus, in order to check whether the word map $w$ is surjective on $SL(2,\mathbb{C})$ it is sufficient to check whether the elements $z$ with $T(z)=\pm2$ are in the image.
		\end{enumerate}
		
	\end{remark}

	\begin{prop}\label{B1}
		The word map induced by the word $w(a,b)=[[a,b],a[a,b]a^{-1}]$ is surjective on $SL(2,\mathbb{C})$.
	\end{prop}
	
	\begin{proof}
		By using the Lemma \ref{1}, it is enough to show that all elements of $SL(2,\mathbb{C})$ with trace $\pm2$ is in the image of $w$. 
		
		First, we will show that all elements in $C\left(
		\begin{bmatrix}
			-1 & 1 \\
			0 & -1
		\end{bmatrix}
		\right) \bigcup C\left(
		\begin{bmatrix}
			1 & 1 \\
			0 & 1
		\end{bmatrix}
		\right)$ is in the image of $w$. From Remark \ref{Remark1}, it enough to show there is an element in image from each conjugacy class. i.e., we have to show that there are matrices $a,b\in SL(2,\mathbb{C})$ such that, \begin{center}
			$w(a,b)=	\begin{bmatrix}
				p_{11} & p_{12} \\
				p_{21} & p_{22}
			\end{bmatrix}$ with $p_{11}+p_{22}=\pm2;  p_{12}\neq0$.
		\end{center}
		We may look for these pairs among the matrices $a=	\begin{bmatrix}
			0 & x \\
			y & z
		\end{bmatrix}$  and $b= 	\begin{bmatrix}
			1 & t \\
			0 & 1
		\end{bmatrix}$.

Consider the polynomial ring $\mathbb{Q}[x,y,z,t]$ over rationals. Let $I$  be the ideal generated by the polynomials $det(a)-1$, $T(A)-2$ and $J$ be the ideal generated by the polynomials $det(a)-1$, $T(A)+2$ in $\mathbb{Q}[x,y,z,t]$. These ideals correspond to the affine subsets $S\subset SL(2,\mathbb{C})^{2}$ and $T\subset SL(2,\mathbb{C})^{2}$ respectively(with identification $(x,y,z,t)$ as $a=	\begin{bmatrix}
			0 & x \\
			y & z  
		\end{bmatrix}$ and $b= 	\begin{bmatrix}
			1 & t \\
			0 & 1 
		\end{bmatrix}$).\\


		The following SageMath computation shows that the polynomial $p_{12}$ is not in the radical of $I$ and $J$. 
		\begin{lstlisting}[language=Python]
			R.<t, x, y, z> = PolynomialRing(QQ, 4)
			X = Matrix(R, 2, 2, [0, x, y, z])
			Y = Matrix(R, 2, 2, [1, t, 0, 1])
			X1 = Matrix(R, 2, 2, [z, -x, -y, 0])
			Y1 = Matrix(R, 2, 2, [1, -t, 0, 1])
			
			C = X * Y * X1 * Y1
			q11, q12 = C[0,0], C[0,1]
			q21, q22 = C[1,0], C[1,1]
			
			C1 = Matrix(R, 2, 2, [q22, -q12, -q21, q11])
			
			D = X * C * X1
			r11, r12 = D[0,0], D[0,1]
			r21, r22 = D[1,0], D[1,1]
			D1 = Matrix(R, 2, 2, [r22, -r12, -r21, r11])
			
			A = C * D * C1 * D1
			TA = A.trace()
			
			p12 = A[0,1]
			
			I = R.ideal([x*y + 1, TA - 2])
			J = R.ideal([x*y + 1, TA + 2])
			
			p12_in_radical_I = p12 in I.radical()
			p12_in_radical_J = p12 in J.radical()
			
			print("Is p12 in radical of I?", p12_in_radical_I)
			print("Is p12 in radical of J?", p12_in_radical_J)
			Is p12 in radical of I? False
			Is p12 in radical of J? False
		\end{lstlisting}

        So the polynomial $p_{12}$ is not vanishing on $S$ and $T$.
		This guarentees the existence of $a,b\in SL(2,\mathbb{C})$ with $w(a,b) \in C\left(
		\begin{bmatrix}
			1 & 1 \\
			0 & 1
		\end{bmatrix}
		\right)$.  Also same for $C\left(
		\begin{bmatrix}
			-1 & 1 \\
			0 & -1
		\end{bmatrix}
		\right)$.
		
		Hence, we have 
		$$C\left(
		\begin{bmatrix}
			-1 & 1 \\
			0 & -1
		\end{bmatrix}
		\right) \bigcup C\left(
		\begin{bmatrix}
			1 & 1 \\
			0 & 1
		\end{bmatrix} \right)\subset \text{Im }w.$$
		Now, to prove surjectivity it is enough to show $-I\in\text{Im }w$. Surjectivity of the word $w$ in $SU(2)$ will give $-I\in \text{Im }w$.
		Hence, the given word map is surjective on $SL(2,\mathbb{C})$.
		\end{proof}

	\begin{prop}\label{B2}
		The word map induced by the word $w(a,b)=[[a,b],b[a,b]b^{-1}]]$ is surjective in $SL(2,\mathbb{C})$.
	\end{prop}
	
	\begin{proof}
    Similar to the Proposition \ref{B1} we need to show that all elements in $C\left(
		\begin{bmatrix}
			-1 & 1 \\
			0 & -1
		\end{bmatrix}
		\right) \bigcup C\left(
		\begin{bmatrix}
			1 & 1 \\
			0 & 1
		\end{bmatrix}
		\right)$ are in the image of $w$. 
We have to show that there are matrices $a,b\in SL(2,\mathbb{C})$ such that, \begin{center}
			$w(a,b)=	\begin{bmatrix}
				p_{11} & p_{12} \\
				p_{21} & p_{22}
			\end{bmatrix}$ with $p_{11}+p_{22}=\pm2;  p_{12}\neq0$.
		\end{center}

		
		We may look for these pairs among the matrices $a=	\begin{bmatrix}
			1 & x \\
			y & 1
		\end{bmatrix}$  and $b= 	\begin{bmatrix}
			w & 1 \\
			0 & z
		\end{bmatrix}$.

		
		Let $I\subset \mathbb{Q}[x,y,z,w]$ be the ideal generated by the polynomials $det(a)-1$, $det(b)-1$, $T(A)-2$ and $J$ be the ideal generated by the polynomial $det(a)-1$, $det(b)-1$, $T(A)+2$. These ideals correspond to the affine subsets $S\subset SL(2,\mathbb{C})^{2}$ and $T\subset SL(2,\mathbb{C})^{2}$ respectively (with identification $(x,y,w,z)$ as $a=	\begin{bmatrix}
			1 & x \\
			y & 1
		\end{bmatrix}$ and $b= 	\begin{bmatrix}
			w & 1 \\
			0 & z
		\end{bmatrix}$).\\

		Following SageMath computation shows that the polynomial $p_{12}$ is non-zero in $S$ and $T$. 
		\begin{lstlisting}[language=Python]
			R.<w, x, y, z> = PolynomialRing(QQ, 4)
			
			X = Matrix(R, 2, 2, [1, x, y, 1])
			Y = Matrix(R, 2, 2, [w, 1, 0, z])
			X1 = Matrix(R, 2, 2, [1, -x, -y, 1])
			Y1 = Matrix(R, 2, 2, [w, -1, 0, z])
			
			C = X * Y * X1 * Y1
			q11, q12 = C[0,0], C[0,1]
			q21, q22 = C[1,0], C[1,1]
			
			C1 = Matrix(R, 2, 2, [q22, -q12, -q21, q11])
			
			D = Y * C * Y1
			r11, r12 = D[0,0], D[0,1]
			r21, r22 = D[1,0], D[1,1]
			D1 = Matrix(R, 2, 2, [r22, -r12, -r21, r11])
			
			TD = D.trace()
			
			p12 = D[0,1]
			
			
			I = R.ideal([x*y, w*z-1, TD - 2])
			J = R.ideal([x*y, w*z-1, TD + 2])
			
			r12_in_radical_I = r12 in I.radical()
			r12_in_radical_J = r12 in J.radical()
			
			print("Is p12 in radical of I?", p12_in_radical_I)
			print("Is p12 in radical of J?", p12_in_radical_J)
			Is p12 in radical of I? False
			Is p12 in radical of J? False
		\end{lstlisting}
		
		This guarentees the existence of $a,b\in SL(2,\mathbb{C})^{2}$ with $w(a,b) \in C\left(
		\begin{bmatrix}
			1 & 1 \\
			0 & 1
		\end{bmatrix}
		\right)$. Also same for $C\left(
		\begin{bmatrix}
			-1 & 1 \\
			0 & -1
		\end{bmatrix}
		\right)$.
		
		Hence, we have 
		$$C\left(
		\begin{bmatrix}
			-1 & 1 \\
			0 & -1
		\end{bmatrix}
		\right) \bigcup C\left(
		\begin{bmatrix}
			1 & 1 \\
			0 & 1
		\end{bmatrix} \right)\subset \text{Im }w.$$
		To establish the surjectivity of the word map, it suffices to show that $-I \in \text{Im }w$. Since the word map $w$ is surjective on $SU(2)$, it follows that $-I$ lies in its image. Therefore, the word map is surjective on $SL(2,\mathbb{C})$.
		\end{proof}
	
	\begin{prop}\label{B3}
		The word map induced by the word $w(a,b)=[[a,b],ab[a,b]b^{-1}a^{-1}]]$ is surjective on $SL(2,\mathbb{C})$.
	\end{prop}
	
	\begin{proof}
		
		We begin by showing that all elements in the conjugacy classes
		\[
		C\left(
		\begin{bmatrix}
			-1 & 1 \\
			0 & -1
		\end{bmatrix}
		\right) \cup 
		C\left(
		\begin{bmatrix}
			1 & 1 \\
			0 & 1
		\end{bmatrix}
		\right)
		\]
		are contained in the image of $w$. According to Remark \ref{Remark1}, it is enough to find one representative from each conjugacy class in the image of $w$. That is, we aim to find matrices $a, b \in SL(2,\mathbb{C})$ such that
		\[
		w(a,b) =
		\begin{bmatrix}
			p_{11} & p_{12} \\
			p_{21} & p_{22}
		\end{bmatrix}, \quad \text{with } p_{11} + p_{22} = \pm 2 \text{ and } p_{12} \neq 0.
		\]
		
		To construct such elements, we consider matrices of the form
		\[
		a = \begin{bmatrix}
			0 & x \\
			y & z
		\end{bmatrix}, \quad 
		b = \begin{bmatrix}
			1 & t \\
			0 & 1
		\end{bmatrix}.
		\]
		
	Let $I \subset \mathbb{Q}[x,y,z,t]$ be the ideal generated by the polynomials corresponding to $\det(a) - 1$ and the trace condition $\text{Tr}(A) - 2$, and let $J$ be the ideal generated by $\det(a) - 1$ and $\text{Tr}(A) +2$. These ideals define affine subsets $S \subset SL(2,\mathbb{C})^2$ and $T \subset SL(2,\mathbb{C})^2$, respectively, where we identify the tuple $(x,y,z,t)$ with the above choices of $a$ and $b$.

		Following SageMath computation shows that the polynomial $p_{12}$ is non-zero on $S$ and $T$. 
		\begin{lstlisting}[language=Python]
			R.<t, b, c, d> = PolynomialRing(QQ, 4)
			X = Matrix(R, 2, 2, [0, x, y, z])
			Y = Matrix(R, 2, 2, [1, t, 0, 1])
			X1 = Matrix(R, 2, 2, [z, -x, -y, 0])
			Y1 = Matrix(R, 2, 2, [1, -t, 0, 1])
			
			C = X * Y * X1 * Y1
			q11, q12 = C[0,0], C[0,1]
			q21, q22 = C[1,0], C[1,1]
			
			C1 = Matrix(R, 2, 2, [q22, -q12, -q21, q11])
			
			D =  X * Y * C  * Y1 * X1
			r11, r12 = D[0,0], D[0,1]
			r21, r22 = D[1,0], D[1,1]
			D1 = Matrix(R, 2, 2, [r22, -r12, -r21, r11])
			
			A = C * D * C1 * D1
			TA = A.trace()
			
			p12 = A[0,1]
			
			I = R.ideal([x*y + 1, TA - 2])
			J = R.ideal([x*y + 1, TA + 2])
			
			p12_in_radical_I = p12 in I.radical()
			p12_in_radical_J = p12 in J.radical()
			
			print("Is p12 in radical of I?", p12_in_radical_I)
			print("Is p12 in radical of J?", p12_in_radical_J)
			Is p12 in radical of I? False
			Is p12 in radical of J? False
		\end{lstlisting}
		
		This guarantees the existence of $a,b\in SL(2,\mathbb{C})$ with $w(a,b) \in C\left(
		\begin{bmatrix}
			1 & 1 \\
			0 & 1
		\end{bmatrix}
		\right)$. Also same for $C\left(
		\begin{bmatrix}
			-1 & 1 \\
			0 & -1
		\end{bmatrix}
		\right)$.
		
		Hence we have 
		$$C\left(
		\begin{bmatrix}
			-1 & 1 \\
			0 & -1
		\end{bmatrix}
		\right) \bigcup C\left(
		\begin{bmatrix}
			1 & 1 \\
			0 & 1
		\end{bmatrix} \right)\subset \text{Im }w.$$
		Now to prove surjectivity it is enough to show $-I\in\text{Im }w$. Surjectivity of the word $w$ in $SU(2)$ will give that $-I\in \text{Im }w$.
		Hence, the given word map is surjective on $SL(2,\mathbb{C})$.
	\end{proof}

		\begin{prop}\label{B4}
		The word map induced by the word $w(a,b)=[[a,b],ab^{2}[a,ab]b^{-2}a^{-1}]$ is surjective on $SL(2,\mathbb{C})$.
	\end{prop}
	
	\begin{proof}
        According to Remark \ref{Remark1}, it is enough to find one representative from each conjugacy class in the image of $w$. That is, we aim to find matrices $a, b \in SL(2,\mathbb{C})$ such that
		\[
		w(a,b) =
		\begin{bmatrix}
			p_{11} & p_{12} \\
			p_{21} & p_{22}
		\end{bmatrix}, \quad \text{with } p_{11} + p_{22} = \pm 2 \text{ and } p_{12} \neq 0.
		\]
		
		To construct such elements, we consider matrices of the form
		\[
		a = \begin{bmatrix}
			x & y \\
			z & 0
		\end{bmatrix}, \quad 
		b = \begin{bmatrix}
			1 & t \\
			0 & 1
		\end{bmatrix}.
		\]
		Let $I \subset \mathbb{Q}[x,y,z,t]$ be the ideal generated by the polynomials corresponding to $\det(a) -1$ and the trace condition $\text{Tr}(A) - 2$, and let $J$ be the ideal generated by $\det(a) -1$ and $\text{Tr}(A) +2$. These ideals define affine subsets $S \subset SL(2,\mathbb{C})^2$ and $T \subset SL(2,\mathbb{C})^2$, respectively, where we identify the tuple $(x,y,z,t)$ with the above choices of $a$ and $b$.

		The following SageMath computation shows that the polynomial $p_{12}$ is non-zero on $S$ and $T$. 
		\begin{lstlisting}[language=Python]
			R.<t, b, c, d> = PolynomialRing(QQ, 4)
			X = Matrix(R, 2, 2, [x, y, z, 0])
			Y = Matrix(R, 2, 2, [1, t, 0, 1])
			X1 = Matrix(R, 2, 2, [0, -y, -z, x])
			Y1 = Matrix(R, 2, 2, [1, -t, 0, 1])
			
			C = X * Y * X1 * Y1
			q11, q12 = C[0,0], C[0,1]
			q21, q22 = C[1,0], C[1,1]
			
			C1 = Matrix(R, 2, 2, [q22, -q12, -q21, q11])
			
			D = X * Y * Y * C * Y1 * Y1 * Y1 * X1
			r11, r12 = D[0,0], D[0,1]
			r21, r22 = D[1,0], D[1,1]
			D1 = Matrix(R, 2, 2, [r22, -r12, -r21, r11])
			
			A = C * D * C1 * D1
			TA = A.trace()
			
			p12 = A[0,1]
			
			I = R.ideal([y*z + 1, TA - 2])
			J = R.ideal([y*z + 1, TA + 2])
			
			p12_in_radical_I = p12 in I.radical()
			p12_in_radical_J = p12 in J.radical()
			
			print("Is p12 in radical of I?", p12_in_radical_I)
			print("Is p12 in radical of J?", p12_in_radical_J)
			Is p12 in radical of I? False
			Is p12 in radical of J? False
		\end{lstlisting}
		This guarentees the existence of $a,b\in SL(2,\mathbb{C})$ with $w(a,b) \in C\left(
		\begin{bmatrix}
			-1 & 1 \\
			0 & -1
		\end{bmatrix}
		\right)$. Also same for $C\left(
		\begin{bmatrix}
			1 & 1 \\
			0 & 1
		\end{bmatrix}
		\right)$.
		
		Hence we have 
		$$C\left(
		\begin{bmatrix}
			-1 & 1 \\
			0 & -1
		\end{bmatrix}
		\right) \bigcup C\left(
		\begin{bmatrix}
			1 & 1 \\
			0 & 1
		\end{bmatrix} \right)\subset \text{Im }w.$$
		To establish the surjectivity of the word map, it suffices to show that $-I \in \text{Im } w$. Since the word map $w$ is surjective on $SU(2)$, it follows that $-I$ lies in its image. Therefore, the word map is surjective on $SL(2,\mathbb{C})$.
	\end{proof}

	\begin{prop}\label{B5}
		The word map induced by the word $w(a,b)=[[a,b],a]$ is surjective on $SL(2,\mathbb{C})$.
	\end{prop}
	
	\begin{proof}
		It is enough to show that all elements of $SL(2,\mathbb{C})$ with trace $\pm2$ is in the image of $w$. 
		Let $a=	\begin{bmatrix}
			t & 0 \\
			0 & -1/t
		\end{bmatrix}$ and $b=	\begin{bmatrix}
			0 & 1 \\
			-1 & 0
		\end{bmatrix}$. Then $w(a,b)=[[a,b],a]=	\begin{bmatrix}
			t^4 & 0 \\
			0 & 1/t^4
		\end{bmatrix}$. Thus, $\begin{bmatrix}
			-1 & 0 \\
			0 & -1
		\end{bmatrix}$ is in the image if we choose $t$ such that $t^4=-1$.\\
		Now, we will show that all elements in $C\left(
		\begin{bmatrix}
			-1 & 1 \\
			0 & -1
		\end{bmatrix}
		\right) \bigcup C\left(
		\begin{bmatrix}
			1 & 1 \\
			0 & 1
		\end{bmatrix}
		\right)$ is in the image of $w$. From Remark \ref{Remark1}, it enough to show there is an element in image from each conjugacy class. i.e., we have to show that there are matrices $a,b\in SL(2,\mathbb{C})$ such that, \begin{center}
			$w(a,b)=	\begin{bmatrix}
				p_{11} & p_{12} \\
				p_{21} & p_{22}
			\end{bmatrix}$ with $p_{11}+p_{22}=\pm2;  p_{12}\neq0$.
		\end{center}
		Take, $a=	\begin{bmatrix}
			x & 0 \\
			 y & z
		\end{bmatrix}$  and $b= 	\begin{bmatrix}
			1 & t \\
			0 & 1
		\end{bmatrix}$.

		
		Let $I\subset \mathbb{Q}[x,y,z,t]$ be the ideal generated by the polynomials $det(a)-1$, $T(A)-2$ and $J$ be the ideal generated by  polynomials $det(a)-1$, $T(A)+2$. These ideals corresponds to the affine subsets $S\subset SL(2,\mathbb{C})^{2}$ and $T\subset SL(2,\mathbb{C})^{2}$ respectively (with identification $(x,y,z,t)$ as $a=	\begin{bmatrix}
			x & 0 \\
			 y & z
		\end{bmatrix}$ and $b= 	\begin{bmatrix}
			1 & t \\
			0 & 1
		\end{bmatrix}$).\\

		Following SageMath computation shows that the polynomial $p_{12}$ is non-zero on $S$ and $T$. 
		\begin{lstlisting}[language=Python]
			R.<x, y, z, t> = PolynomialRing(QQ, 4)
			
			X = Matrix(R, 2, 2, [x, 0, y, z])
			Y = Matrix(R, 2, 2, [1, t, 0, 1])
			X1 = Matrix(R, 2, 2, [z, 0, -y, x])
			Y1 = Matrix(R, 2, 2, [1, -t, 0, 1])
			
			C = X * Y * X1 * Y1
			q11, q12 = C[0,0], C[0,1]
			q21, q22 = C[1,0], C[1,1]
			
			C1 = Matrix(R, 2, 2, [q22, -q12, -q21, q11])
			
			D = C * X * C1 * X1
			r11, r12 = D[0,0], D[0,1]
			r21, r22 = D[1,0], D[1,1]
			D1 = Matrix(R, 2, 2, [r22, -r12, -r21, r11])
			
			TD = D.trace()
			
			p12 = D[0,1]
			
			
			I = R.ideal([x*y-1, TD - 2])
			J = R.ideal([x*y-1, TD + 2])
			
			p12_in_radical_I = p12 in I.radical()
			p12_in_radical_J = p12 in J.radical()
			
			print("Is p12 in radical of I?", p12_in_radical_I)
			print("Is p12 in radical of J?", p12_in_radical_J)
			Is p12 in radical of I? False
			Is p12 in radical of J? False
		\end{lstlisting}
		
		This guarentees the existence of $a,b\in SL(2,\mathbb{C})$ with $w(a,b) \in C\left(
		\begin{bmatrix}
			-1 & 1 \\
			0 & -1
		\end{bmatrix}
		\right)$. Also same for $C\left(
		\begin{bmatrix}
			1 & 1 \\
			0 & 1
		\end{bmatrix}
		\right)$.
		
		Hence we have 
		$$C\left(
		\begin{bmatrix}
			-1 & 1 \\
			0 & -1
		\end{bmatrix}
		\right) \bigcup C\left(
		\begin{bmatrix}
			1 & 1 \\
			0 & 1
		\end{bmatrix} \right)\subset \text{Im }w.$$
		This proves the surjectivity of the word map corresponds to the given word .
		\end{proof}
	\textit{Proof of Theorem \ref{B}}. The Proposition \ref{B1} give surjectivity of the word $w_{1}$, Proposition \ref{B2} shows the surjectivity of the word $w_{2}$, Proposition \ref{B3} shows the surjectivity of the word $w_{3}$,  Proposition \ref{B4} shows the surjectivity of the word $w_{4}$, Proposition \ref{B5} shows the surjectivity of the word $w_{5}$.
  \subsection*{Acknowledgment:} 
		It is a pleasure to thank my advisor Anirban Bose for his suggestions and comments on the manuscripts. The author would also like to thank T. Bandman and Y. G. Zarhin for some helpful discussions and clarification of doubts. 

	\bibliographystyle{alpha}
	\bibliography{Word}

\end{document}